\documentclass[arxiv,reqno,twoside,a4paper,12pt]{article}
\usepackage[hidelinks,colorlinks=true,linkcolor=blue,citecolor=blue]{hyperref}
\usepackage{tocloft}

\usepackage[driver=pdftex,margin=3cm,top=3cm, bottom=3cm,centering]{geometry}
\usepackage{mathrsfs}
\usepackage{tikz}

\usepackage[para]{footmisc}

\usepackage[indentafter]{titlesec}
\titleformat{name=\section}{}{\thetitle.}{0.8em}{\centering\scshape}
\titleformat{name=\subsection}[runin]{}{\thetitle.}{0.5em}{\bfseries}[.]
\titleformat{name=\subsubsection}[runin]{}{\thetitle.}{0.5em}{\itshape}[.]
\titleformat{name=\paragraph,numberless}[runin]{}{}{0em}{}[.]
\titlespacing{\paragraph}{0em}{0em}{0.5em}
\titleformat{name=\subparagraph,numberless}[runin]{}{}{0em}{}[.]
\titlespacing{\subparagraph}{0em}{0em}{0.5em}

\usepackage[utf8]{inputenc}\usepackage{amsmath,amssymb,latexsym,amsfonts,amsthm}
\usepackage{float}\floatstyle{plaintop}\restylefloat{table}
\usepackage[tableposition=top]{caption}
\usepackage{graphicx}\usepackage{pstricks}
\usepackage{setspace}
\usepackage{epsfig}
\usepackage{multirow,paralist,listings,longtable}
\usepackage{marginnote,sidenotes}
\def\taskip{\renewcommand{\arraystretch}{1}\renewcommand{\baselinestretch}{1}}
\def\teskip{\renewcommand{\arraystretch}{1.1}\renewcommand{\baselinestretch}{1.15}}
\def\bexa{\begin{example}}\def\eexa{\end{example}}
\def\brem{\begin{remark}}\def\erem{\end{remark}}
\def\bthm{\begin{theorem}}\def\ethm{\end{theorem}}
\def\blem{\begin{lemma}}\def\elem{\end{lemma}}
\def\bcor{\begin{corollary}}\def\ecor{\end{corollary}}
\def\bdefi{\begin{definition}}\def\edefi{\end{definition}}
\def\bexe{\begin{exercise}}\def\eexe{\eex\end{exercise}}

\def\rb{\raisebox}

\newcommand{\abs}[1]{\lvert#1\rvert}\newcommand{\norm}[1]{\lVert#1\rVert}

\def\bmip{\begin{minipage}{\textwidth}}\def\emip{\end{minipage}}
\def\huga#1{\begin{gather} #1 \end{gather}}

\def\hual#1{\begin{align} #1 \end{align}}

\newcommand{\R}{{\mathbb R}}
\newcommand{\N}{{\mathbb N}}
\newcommand{\Z}{{\mathbb Z}}

\def\CO{{\cal O}}

\def\ga{\gamma}
\def\ds{\displaystyle}
\def\pa{{\partial}}\def\lam{\lambda}
\newcommand{\bi}{\begin{itemize}}\newcommand{\ei}{\end{itemize}}
\newcommand{\ben}{\begin{enumerate}}\newcommand{\een}{\end{enumerate}}
\newcommand{\bce}{\begin{center}}\newcommand{\ece}{\end{center}}
\newcommand{\bci}{\begin{compactitem}}\newcommand{\eci}{\end{compactitem}}
\newcommand{\bcen}{\begin{compactenum}}\newcommand{\ecen}{\end{compactenum}}
\newcommand{\bcena}{\begin{compactenum}[(a)]}
\newcommand{\reff}[1]{(\ref{#1})}
\newcommand{\btab}[2]{\begin{tabular}{#1}#2\end{tabular}}

\newcommand{\spr}[1]{\langle #1 \rangle}
\newcommand{\hs}[1]{{\hspace{#1}}}\newcommand{\vs}[1]{{\vspace{#1}}}

\def\eps{\varepsilon}

\newcommand{\barr}{\begin{array}}\newcommand{\earr}{\end{array}}
\newcommand{\bpm}{\begin{pmatrix}}\newcommand{\epm}{\end{pmatrix}}
\newcommand{\bsm}{\left(\begin{smallmatrix}}
\newcommand{\esm}{\end{smallmatrix}\right)}

\newcommand{\bsmm}{\begin{large}\left(\begin{smallmatrix}}
\newcommand{\esmm}{\end{smallmatrix}\right)\end{large}}
\def\sm{\small}

\newcommand{\ba}{\begin{array}}\newcommand{\ea}{\end{array}}
\def\dd{\, {\rm d}}

\def\Om{\Omega}

\def\del{\delta}

\def\eex{\hfill\mbox{$\rfloor$}}

\def\Del{\Delta}

\def\sig{\sigma}
\def\al{\alpha}

\def\Ga{\Gamma}

\def\bd{\begin{displaymath}} \def\ed{\end{displaymath}}
\def\ba{\begin{array}} \def\ea{\end{array}}
  
\def\eps{\varepsilon}

\def\ind{{\rm ind}}\def\ig{\includegraphics}

\def\pdep{{\tt pde2path}}

\def\tew{\textwidth}

\usepackage[mathbf]{euler}

\numberwithin{equation}{section}
\usepackage{chngcntr}\usepackage{verbatim}

\def\hutab#1#2{\taskip\begin{\table}#1\end{table}\teskip}
\allowdisplaybreaks
\usepackage{hyperref}

\newcommand{\cupdot}{\mathbin{\mathaccent\cdot\cup}}

\newcommand{\vol}[0]{\text{vol}}\newcommand{\diam}{\text{diam}}
\newcommand{\Lip}{\text{Lip}}\newcommand{\loc}{\text{loc}}

\newcommand{\cE}{\mathcal E}
\newcommand{\cH}{\mathcal H}

\newcommand{\cF}{\mathcal F}

\newcommand{\Mbar}{\overline{M}}
\newcommand{\Vtilde}{\tilde{V}}
\newcommand{\Gtilde}{\tilde{G}}

\def\xti{\tilde x}\def\yti{\tilde y}\def\zti{\tilde z}
\newtheorem{theorem}{Theorem}[section]
\newtheorem{remark}[theorem]{Remark}
\newtheorem{conj}[theorem]{Conjecture}
\newtheorem{definition}[theorem]{Definition}
\newtheorem{example}[theorem]{Example}

\newtheorem{Lem}[theorem]{Lemma}

\newtheorem{Prop}[theorem]{Proposition}
\newtheorem{resu}[theorem]{Result}
\newtheorem{Def}[theorem]{Definition}
\newtheorem{Rem}[theorem]{Remark}

\def\sg{\sqrt{g}}\def\sgi{\frac 1 {\sqrt{g}}}
\def\mid{:}  

\title{Phase transitions and minimal interfaces on 
manifolds with conical singularities\\ {\normalsize \today}}

\author{Daniel Grieser${}^*$, \ Sina Held${}^\dagger$, \ Hannes Uecker${}^\ddagger$ and Boris Vertman${}^\sharp$}
\date{\vspace{-5ex}}

\begin{document}
\maketitle

\begin{abstract} 
Using $\Ga$--convergence, we study the 
Cahn-Hilliard problem with interface width parameter 
$\eps{>}0$ for phase transitions on 
manifolds with conical singularities. 
We prove that minimizers of the corresponding 
energy functional exist and converge, as $\varepsilon{\to}0$, to a function that takes only two values with an interface along 
a hypersurface that has minimal area among those satisfying a volume constraint. 
In a numerical example, we use continuation 
and bifurcation methods to study families of critical points at small $\varepsilon{>}0$ on 2D elliptical
cones, parameterized by height and ellipticity of the base.  
Some of these critical points are minimizers with interfaces crossing the cone tip.
On the other hand, we prove that interfaces which are 
minimizers of the perimeter functional, corresponding to $\eps=0$, 
never pass through the cone tip for general cones with angle less than $2\pi$.
Thus tip minimizers for finite $\eps{>}0$ must become saddles as $\eps{\to}0$, 
and we numerically identify the associated bifurcation, finding 
a delicate interplay of $\eps{>}0$ and the cone parameters in our example. 
\end{abstract}

\setcounter{tocdepth}{1}
\tableofcontents

\section{Introduction and statement of the main results}

\subsection{The Cahn-Hilliard problem}\label{CH-intro-subsection} 
In the Cahn-Hilliard gradient theory of phase transitions we are interested in 
the minimizers, or more generally critical points, of the energy functional
\huga{\label{eemin} 
E_\eps(u) :=\frac 1 {2\sig}\int_M \frac \eps 2 |\nabla u|_g^2 + \frac 1 \eps W(u)\dd \textup{vol}_g,
} 
where $\varepsilon > 0$ is a parameter, $(M,g)$ is a possibly incomplete Riemannian manifold of dimension $d$,
 of finite volume $\textup{vol}_g(M)$, with boundary $\partial M$,
and we have used the notation $W$ for a {\em double well potential}, 
e.g.~$W(x) := \frac 1 4(x^2-1)^2$ for $x \in \R$ (this 
can be replaced by any coercive $C^1$ function with 
exactly two minima), 
% at $x_0<x_1$), 
and the normalization 
$$
\sig:=\int_{-1}^1\sqrt{W(x)/2}\dd x=\frac{\sqrt{2}}{3}.
$$
We minimize among real-valued $u \in H^1(M) \cap L^4(M)$, subject to
Neumann boundary conditions $\pa_\nu u{=}0$  on the boundary $\pa M$ (with normal unit vector field $\nu$) 
and under the mass constraint
\begin{align}\label{mass}
\spr{u}:=\frac 1 {\textup{vol}_g(M)}
\int_M u \dd \textup{vol}_g=m,
\end{align}
where $m\in[-1,1]$ is a \emph{prescribed mass}, often 
set to $m=0$ below. 
Using the Lagrange multiplier $\lam$ we introduce the Lagrangian 
\begin{equation}\label{lagrangian}
L(u,\lam) = E_\eps(u)+\lam\bigg(\spr{u}-m\bigg).
\end{equation}
Then the first variations of $L(u,\lam)$ with respect to $u$ and $\lam$ 
yield the Euler--Lagrange equations as necessary first order minimization conditions, 
\begin{align}\label{ch2}\begin{split}
&{\rm (a)}\quad -\eps^2\Del u+W'(u)-\lam=0 \text{ in }M,\quad 
\pa_\nu u=0 \text{ on }\pa M, \\ 
&{\rm (b)}\quad q(u)=0, \quad\text{ where } q(u):=\spr{u}-m, 
\end{split} \end{align}
where $\Del$ is the (negative)
Laplace--Beltrami operator associated to $(M,g)$. The PDE (a) with $\lam=0$ 
%(which automatically holds at $m=0$)
 is also called Allen--Cahn 
equation. The trivial (spatially homogeneous) solution branch of 
\reff{ch2} is $u\equiv m$, with 
Lagrange multiplier $\lam=W'(m)$. 

\subsection{Relation to minimal hypersurfaces} The problem \reff{eemin} is closely related to constant mean curvature surfaces in $M$, i.e.\ surfaces which have minimal area among those satisfying a volume constraint. 
The pure phases $u=\pm 1$ minimize the double well
potential $W$ and hence also 
the energy functional $E_\eps$ with 
the mass constraint $m=\pm 1$, respectively. However, for $|m|\ne 1$ 
these minimizers do not fulfill the mass constraint (\ref{ch2}b), 
and instead we expect 
regions of pure 
phases $u=1$ and $u=-1$ separated by transition regions or 
interfaces with $u\in(-1,1)$. 
The term $\frac \eps 2 |\nabla u|_g^2$ models an interface energy 
density, and for $\eps > 0$ the minimizers $u_\eps$ are  
smooth.  However, from the $\eps$--scaling together with the scaling 
of the potential energy 
$\frac 1 \eps W(u)$ we expect the interfaces to be steep  
and of width $\CO(\eps)$. 
This suggests that suitable sequences of minimizers
$u_\eps$ of $E_\eps$ for $\eps\to 0$ 
converge to a function $u_0$ which only takes values in the \emph{pure 
phases} $u=\pm 1$, and such that the interface 
\begin{align}\label{interface}
I_0=\pa\{u_0{=}-1\} \subset M \backslash \partial M
\end{align}
(we take the boundary 
in $M \backslash \partial M$, so that $\partial M$ is not part of the interface)
has minimal $(d{-}1)$-dimensional volume
among those satisfying the mass constraint \reff{ch2}(b).\footnote{The mass constraint amounts to $\frac{\vol\{u=1\}}{\vol\{u=-1\}} = \frac{1+m}{1-m}$ for $|m|<1$.} 

The above heuristics is proven in \cite{Mo} for domains in Euclidean 
space,  
 with the following precise statement where henceforth %for simplicity we restrict to the case $m=0$ and 
 we write\medskip
\bci
\item \emph{length} for the $(d{-}1)$-dimensional volume,
\item \emph{area} for the $d$-dimensional volume.
\eci

\begin{theorem}\label{acms} Let $\Omega\subset\R^d$ be a bounded domain with Lipschitz boundary.
Let $(u_\eps)$ be a sequence of minimizers of $E_{\eps}$ with $\eps \to 0$,
subject to $q(u_\eps)=0$ % with $m=0$
with $|m|<1$. Then there exists a subsequence of $(u_\eps)$ 
that converges in $L^1(\Om)$ to a function $u_0$ which only takes values in $\pm 1$, 
with the interface $I_0$ as in \eqref{interface} having minimal length $|I_0|$ among those satisfying the mass constraint. Moreover, for that subsequence
$$
\lim_{\eps \to 0} E_{\eps}(u_\eps) = |I_0|.
$$
\end{theorem}

Results of this type have been extended and refined,  
and have been transferred to Cahn--Hilliard problems on (smooth) Riemannian manifolds, including min--max type results for 
critical points of $E_\eps$ (saddle--points), see, e.g., \cite{GHP03, Ton05, Pa12, GG, BNAP22, HuTo00}. % and many more. 
A standard setting for this is $\Ga$--convergence, 
already discussed in \cite{Mo}, see \cite[\S 13]{Ri} for a textbook presentation.

\subsection{Main results: a) Convergence of minimizers} 
Our first main result is the transfer of the 
convergence of minimizers results as in Theorem \ref{acms} to the case 
of compact manifolds with boundary and conical singularities, using 
results from geometric measure theory \cite{Morgan, Fe14}. 
These spaces, denoted $\Mbar$, with regular part the Riemannian manifold $(M,g)$, are defined in Appendix \ref{coneHD}, and we show as a combination of Propositions \ref{existence-prop}, 
\ref{regularity-prop} and Theorem \ref{thm:1} the following. 
As before, we fix a mass satisfying $|m|<1$.

\bthm\label{mthm1} 
Let $\Mbar$ be a compact manifold with boundary and  
finitely many conical singularities. Then the following holds
\begin{enumerate}
\item Minimizers $u_\eps$ of $E_{\eps}$ with $q(u_\eps)=0$ and satisfying Neumann boundary conditions at $\partial M$
exist for $\eps>0$ and are strong solutions to the Allen-Cahn equation \eqref{ch2}. 
\item $E_\eps$ $\Gamma$-converges to $E_0$ as $\eps\rightarrow0^+$ with respect to the strong $L^1$-topology, 
where $E_0$ is the perimeter functional, see \reff{E_0def} for the precise 
definition.
In particular, if $(u_\eps)$ is a sequence of such minimizers for $\eps \to 0$, 
and $u_\eps\to u_0$ in $L^1(M)$, then $u_0$ only takes values in $\pm 1$, 
the interface $I_0$ as in \eqref{interface} has minimal length 
$|I_0|$ among the hypersurfaces satisfying the mass constraint, and 
\begin{align}\label{energy-length}
\lim_{\eps \to 0} E_{\eps}(u_\eps) = |I_0|.
\end{align}
\end{enumerate}
\ethm 

\brem\label{ueexrem}{\rm a) The convergence of a subsequence  
$u_\eps\to u_0$ in $L^1(M)$ can be obtained under rather general 
conditions, see \cite[Proposition 3]{Mo} for the Euclidean case, namely if 
(a) $(u_\eps(x))$ is bounded in $L^\infty$, or (b) under 
natural growth conditions for $W$, for instance fulfilled by our 
prototype $W$. The proof from \cite{Mo} transfers 
directly to manifolds as in Thm \ref{mthm1}, and similar for a proof 
of (a) from \cite{GM88} under different assumptions on $W$. 
Thus, in Theorem \ref{mthm1} 
we mainly assume $u_\eps\to u_0$ in $L^1(M)$ for simplicity. 

b) In the smooth case  \reff{energy-length} is obtained in \cite{GG}  also
for sequences $u_\eps$ which have Morse index 
$\ind(u_\eps)=1$ for all $\eps$, i.e., saddle points of the 
energy. 
}\erem 

\subsection{Main results: b) Minimizers on conical surfaces} 

Our second main result studies minimizers of $E_0$ on surfaces with boundary and conical singularities, 
namely in Proposition \ref{prop:tip interface} we prove the following. See Appendix \ref{coneHD} for the definition of the angle of a conical singularity.

\begin{Prop} \label{prop:tip interface-intro}
Let $\Mbar$ be a surface with boundary and a conical singularity $P$ of angle $\al<2\pi$. 
Then any curve of minimal length which divides $\Mbar$ into two parts of prescribed  area ratio does not pass through $P$. 
\end{Prop}
\subsection{Main results: c) Numerical study of minimizers} 

Our third contribution is a numerical study of critical points of $E_\eps$ 
on 2D cones, see Fig.~\ref{f1} for a preview of 
``typical'' solutions on a ``typical'' cone at 
$\eps{=}0.1$. Here we restrict to $m=0$.
Numerically, the problem is 
best considered by continuation and bifurcation: we first fix 
$\eps>0$ and aim to obtain a selection of solutions $u_\eps$ 
at $m=0$, by bifurcation of nonhomogeneous 
solutions $u_\eps$ from the 
homogeneous branch $u\equiv m$ and continuation to $m=0$. 
Subsequently, we consider continuation in $\eps\to 0$, aiming to 
identify the limiting interfaces $I_0$, and to check the formula 
\eqref{energy-length}.%
\footnote{A similar numerical analysis is performed e.g.~in 
\cite[\S6.9 and \S10.1]{p2p} 
over 2D and 3D flat and curved (non-singular) manifolds, also including 
the case of finite Morse index saddle points, for which we numerically 
obtain the same convergence as for mininizers as in Theorem \ref{acms}, 
see also Remark \ref{ueexrem}(b).}\medskip

In addition to the parameters $m$ and $\eps$, in our numerics 
we consider elliptic cones of height $h>0$ with short semi-axis 1 and 
long semi-axis $a\ge 1$.  This yields rather rich bifurcation 
diagrams and different types of interfaces as seen in the numerical 
examples in Fig.~\ref{f1}.

\begin{figure}[h]
\bce
\btab{lll}{{\sm (a) T1}&{\sm (b) T2}&{\sm (c) T3}\\
\hs{-2mm}\ig[width=0.3\tew]{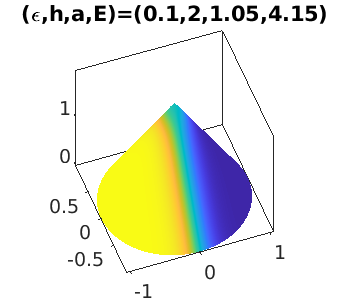}&
\hs{2mm}\ig[width=0.3\tew]{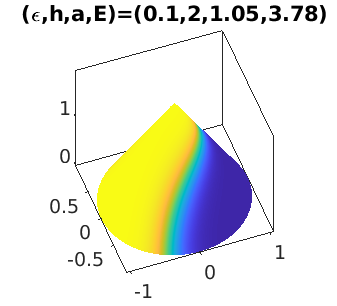}&
\hs{2mm}\ig[width=0.3\tew]{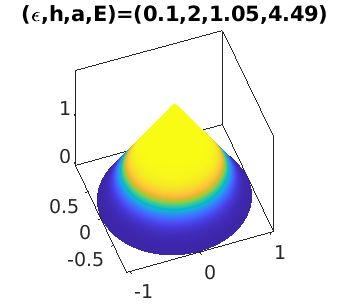}
}
\ece
\vs{-2mm}
\caption{{\small Three basic types T1, T2 and T3 
of interfaces on a
cone of height $h=2$, ellipticity $a=1.05$ (almost circular); 
$\eps=0.1$, energy $E=E_\eps$ as given, approximating the 
interface length. The {\em tip interface} T1 is 
a saddle point here (but for fixed small $\eps>0$ 
becomes a global minimizer on a sufficiently 
elliptic cone); T2 (winding) is the global minimizer, and  
T3 (roughly horizontal) is a local minimizer; see \S\ref{numsec} for 
details.  }\label{f1}}
\end{figure}

\begin{enumerate}
\item \textbf{\emph{Type T1 interfaces} passing through the conical tip (\lq tip-interfaces\rq):}

We find that for any $\eps>0$ and $h>0$ 
there is a (large) 
$a_0\ge 1$ such that for all $a>a_0$ the minimizer of $E_\eps$ shows 
an  interface going 
through the tip of the cone 
along the short semi--axis, i.e., of length $2\sqrt{1+h^2}$. 
%Henceforth we call such an interface a {\em tip--interface}, 
%or type T1 interface.
For instance Fig.~\ref{f1}(a), at fixed $\eps=0.1$, 
shows such a tip--interface 
$u_\eps$ with 
$\ind(u_\eps)=1$, i.e., a saddle point for $E_\eps$ (since $a=1.05$ is small here).

\item \textbf{\emph{Type T2 interfaces} winding around the conical tip:}

Fig.~\ref{f1} (b) shows a local (and global) minimizer of type T2.

\item \textbf{\emph{Type T3 interfaces} running horizontally around the tip:}

Fig.~\ref{f1} (c) shows another saddle point of 
type T3 with a roughly ``horizontal'' interface, and we expect 
T3 solutions to become minimizers at large $h$ (for small $\eps>0$, 
and also in the limit $\eps\to 0$). 
\end{enumerate}

In \S\ref{numsec} we present the numerical computations partly 
previewed in Fig.\ref{f1}. Naturally, the limit $\eps{\to}0$ 
is of particular interest for sequences of tip--interfaces, while 
for (sequences of) interfaces which avoid the tip 
as in Fig.~\ref{f1}(b,c) we are essentially back to the non--singular case. 

The angles of our elliptic cones are always less than $2\pi$, so tip--interfaces are {\em never} 
minimizers at $\eps=0$ by Proposition \ref{prop:tip interface-intro}. Therefore, in the numerical continuation 
in $\eps\to 0$ 
for T1 interfaces which {\em are} minimizers at some starting $\eps_0>0$ 
we find an $0<\eps_1<\eps_0$ such that at $\eps_1$ a branch of 
minimizing T2 interfaces bifurcates from the T1 branch. 
Nevertheless, the sequence of tip--interfaces, 
unstable at sufficently small $\eps>0$, 
converges for $\eps\to 0$ to the expected limit (tip) interface and 
\reff{energy-length} holds, i.e., we believe that Remark \ref{ueexrem}(b) 
also  holds in our case of manifolds with conical singularities. We can phrase
it as a conjecture.

\begin{conj}\reff{energy-length} also holds for sequences of
saddle points of energy, converging to an interface $I_0$, passing possibly
through the conical singularity. 
\end{conj}

Additionally, our numerics suggest that for the limit $\eps\to 0$ 
the main influence of the tip is that for tip--interfaces 
the convergence in \reff{energy-length} is {\em slower} than 
for  interfaces which avoid the tip. Correspondingly, 
 level-lines such as $u_\eps=\pm \frac 1 2$ at small finite $\eps>0$ 
keep a larger distance from the interface $u_\eps=0$ near the tip than 
in smooth parts of the cones. Moreover, this effect becomes stronger for 
more pointed cones, i.e., can be seen as a measure of 
the strength of the singularity.

\subsection{Some related problems}
By Cahn--Hilliard {\em problem} we denote the elliptic equation (\ref{ch2}a) 
together with the mass constraint (\ref{ch2}b). The dynamic Cahn--Hilliard {\em equation} (in Euclidean space) 
is the mass conserving flow 
\huga{\label{ch4}
\pa_t u=\nabla\cdot[\nabla\del_u E_\eps(u)]=-\Delta[\eps^2\Delta u-W'(u)], 
}
also called $H^{-1}$ gradient flow, 
where $\del_u$ denotes the variational derivative. 
For zero-flux boundary conditions, i.e., $\pa_n u=\pa_n\Delta u=0$ on $\pa\Om$, 
this conserves the mass $\int_\Om u\dd x$, and steady states of \reff{ch4} 
fulfill $\eps^2\Delta u-W'(u)=\lam$ for some $\lam\in\R$ and hence 
are solutions of \reff{ch2}.  

See \cite{Ell89} for basic results 
on existence of solutions of \reff{ch4} (in 2D flat domains), 
their numerical approximation, 
and their basic dynamical behavior, which can roughly be characterized as 
follows: Starting from essentially random initial data (with mass $0$), 
the solution rapidly evolves to a fine grained structure with complex 
interfaces between the phases $u=\pm 1$, also aptly called ``fat spaghettis''. 
After this initial phase, a slow coarsening process sets in, during 
which interfaces move and disappear (regions of pure phases $u\approx 
1$ or $u\approx -1$ coming together), on longer and longer time scales. 
See also  \cite{Mira19, DF20} for comprehensive reviews of 
other Cahn--Hilliard type 
equations used to describe diffusive interfaces in a variety of 
settings and applications, and of their analytical and numerical treatment. 

The parabolic Allen-Cahn equation is given by
$$
\partial_t u -\eps^2\Del u+W'(u) = 0, 
$$
again with Neumann boundary conditions at $\partial M$. 
For small $\eps$, the level sets of $u$ concentrate around an interface that
evolves in time under a generalized mean curvature flow \cite{ESS92}. However, 
the mass is in general not conserved. This holds for 
$
\partial_t u -\eps^2\Del u+(W'(u)-\spr{W'(u)}) = 0, 
$
for which solutions converge to a volume preserving mean curvature flow 
\cite{CHL10}. 
It should be interesting to transfer such results to the case 
of manifolds with conical singularities too. See 
\cite{RS13, BV16} for well--posedness results for the Allen--Cahn 
and Cahn--Hilliard 
equations on singular manifolds. 

\subsection{Structure of the paper} 

In \S \ref{section-proof} we prove Theorem \ref{mthm1}; in \S \ref{Danielsec}
we study length-minimizing area halving--curves
%for $\varepsilon = 0$
on surfaces with conical singularities and prove 
% on families of 2D elliptic cones and in particular 
%show by elementary analysis that these curves cannot pass through the cone tip, 
%as asserted in 
Proposition \ref{prop:tip interface-intro};
in \S\ref{numsec} we perform numerical bifurcation analysis on 
cones at finite $\eps>0$, and then let $\eps\to 0$. 
In  Appendix \ref{section-appendix} 
we collect some auxiliary analytical results needed 
for the proof of Thm \ref{mthm1}, and in Appendix \ref{coneHD} we collect some basic facts on conical singularities. \medskip

\noindent \emph{Acknowledgments:} The authors thank Marco Guaraco for valuable 
comments.

\section{Convergence of minimizers on singular spaces}\label{section-proof}

Throughout, let $\Mbar$ be a compact manifold with boundary and  finitely many conical singularities, and denote by $M$ its regular part, with Riemannian metric $g$, see Definition \ref{def:mfd con}.

\subsection{Existence and regularity of solutions}

The results of this subsection in fact hold for general incomplete 
Riemannian manifolds $(M,g)$ of finite volume. 
This goes well beyond compact manifolds with boundary and conical singularities. 
Note that, when we talk about manifolds with boundary here then the letter $M$ denotes the interior, not including the boundary.

However, all the other subsections require the singularities to be conical and the space to be compact.

\subsubsection{Self-adjoint extensions of the Laplacian}
Consider the gradient $\nabla$ on $(M,g)$, mapping smooth functions $C^\infty_0(M)$
to smooth sections of the tangent bundle $C^\infty_0(M,TM)$. We write $\Delta = - \nabla^t \nabla$ for the (negative)
Laplace Beltrami operator. If $M$ is a closed compact manifold, then the usual Sobolev spaces
\begin{align*}
&H^1(M):= \{u \in L^2(M) \mid \nabla u \in L^2(TM)\}, \\
&H^2(M):=\{u \in H^1(M) \mid \nabla^2 u \in L^2(\Lambda^2 TM)\},
\end{align*}
define unique closed, and in the latter case unique self-adjoint, extensions of 
$\nabla$ and $\Delta$, respectively. If $(M,g)$ is non-compact, the extensions may not longer
be unique and we shall now employ the formalism of minimal and maximal 
extensions, as in the seminal work by Br\"uning-Lesch \cite{Hilbert-complexes}. 

The maximal and minimal closed (with respect to the graph norm)
extensions $\nabla_{\max}, \nabla_{\min}$ are defined by the respective domains
\begin{equation}\label{minmax-def}\begin{split}
&\mathscr{D}(\nabla_{\max}) := \{u \in L^2(M) \mid \nabla u \in L^2(TM)\} = H^1(M), \\
&\mathscr{D}(\nabla_{\min}) := \{u \in \mathscr{D}(\nabla_{\max}) \mid \exists \{u_n\}_n \subset C^\infty_0(M): 
\\ & \qquad  \qquad \qquad u_n \xrightarrow{n\to \infty} u, \nabla u_n  \xrightarrow{n\to \infty} \nabla u \ \textup{in} \ L^2\}.
\end{split}\end{equation}
The two extensions define ideal boundary conditions in the sense of Cheeger \cite{Che2} and Br\"uning-Lesch \cite{Hilbert-complexes} and yield
self adjoint extensions of the Laplace Beltrami operator on $(M,g)$
\begin{equation}\begin{split}
&\Delta_D := - \nabla^*_{\min}\nabla_{\min}, \quad \mathscr{D}(\Delta_D) = \mathscr{D} (\nabla^*_{\min}\nabla_{\min}), \\
&\Delta_N := - \nabla^*_{\max}\nabla_{\max}, \quad \mathscr{D}(\Delta_N) = \mathscr{D} (\nabla^*_{\max}\nabla_{\max}).
\end{split}\end{equation}
The two extensions define the well-known Dirichlet and Neumann boundary conditions in case  $M$ is the interior of a compact Riemannian manifold with boundary, $\Mbar=M\cup \partial M$. We shall be precise: the trace theorem asserts that the obvious restriction, 
mapping any $u \in C^\infty(\Mbar)$ to $u|_{\partial M} \in C^\infty(\partial M)$
admits a continuous extension 
\begin{align}\label{trace-mapping}
\textup{tr}: \mathscr{D}(\nabla_{\max}) \to L^2(\partial M),
\end{align}
where continuity holds with respect to the graph norm on $\mathscr{D}(\nabla_{\max})$
and the $L^2$ space on the right is defined with respect to the restriction of $g$ to $\partial M$.
Similarly, if $\partial_\nu$ is the unit normal inward vector field on $\partial M$, 
extended smoothly to the interior, then $u \mapsto \partial_\nu u \restriction \partial M$
admits a continuous extension 
\begin{align}
\textup{tr} \circ \partial_\nu: \mathscr{D}(\nabla_{\max}) \to H^{-1/2}(\partial M).
\end{align}
By continuity of the trace we observe
\begin{equation}
 \mathscr{D}(\nabla_{\min}) \subseteq \{ u \in \mathscr{D}(\nabla_{\max}) = H^1(M) \mid \textup{tr} \, u = 0\} =: H^1_0(M).
\end{equation}
The following result is observed in \cite[\S 4]{Hilbert-complexes}. 
If $\Mbar$ is  non-compact, e.g.\ it has 'interior singularities', and if $\phi \in C^\infty(\Mbar)$ is compactly supported, i.e.\ supported  \emph{away} from the singularities in the interior, then
\begin{equation}\label{DNH}\begin{split}
& \phi \mathscr{D}(\Delta_D) = \phi \{ u \in H^2(M) \mid \textup{tr} \, u = 0\}, \\
& \phi \mathscr{D}(\Delta_N) = \phi \{ u \in H^2(M) \mid \textup{tr} \circ \partial_\nu \, u = 0\}.
\end{split}\end{equation}
This motivates the choice of the subscripts D and N that stand for Dirichlet and Neumann 
boundary conditions. 

\brem\label{laprem}{\rm See Appendix \ref{coneHD} for further comments 
on $\mathscr{D}(\Delta_D)$ and $\mathscr{D}(\Delta_N)$ in the conical case.}
\erem

\subsubsection{Existence of minimizers}

We begin with the existence of minimizers. For coercive 
functionals on $\R^d$, this is a classical result, see 
e.g.~\cite[Chapter 8]{evans}. The
proof presented here holds on any finite volume Riemannian manifold,
possibly incomplete, including those with conical singularities.
\begin{Prop}\label{existence-prop}
Consider the 
function spaces
\begin{equation}\begin{split}
&\mathscr{A}_D:= \{ u \in \mathscr{D}(\nabla_{\min}) \cap L^4(M) \mid \spr{u} = m\}, \\
&\mathscr{A}_N:= \{ u \in \mathscr{D}(\nabla_{\max}) \cap L^4(M) \mid \spr{u} = m\}.
\end{split}
\end{equation}
Then for any fixed $\eps > 0$ there exist minimizers $u^\eps_D \in \mathscr{A}_D$ and $u^\eps_N \in \mathscr{A}_N$ such that 
\begin{align}
E_\eps(u^\eps_D) = \inf_{u \in \mathscr{A}_D} E_\eps(u), \quad 
E_\eps(u^\eps_N) = \inf_{u \in \mathscr{A}_N} E_\eps(u).
\end{align}
\end{Prop}

\begin{proof}
We prove existence of a minimizer $u_D \equiv u_D^\eps \in \mathscr{A}_D$. The argument 
for the other case is exactly the same, with $u^\eps_D, \nabla_{\min}, \mathscr{A}_D$
replaced by $u^\eps_N, \nabla_{\max}, \mathscr{A}_N$, respectively. 
First, consider a sequence $(u_n) \subset \mathscr{A}_D$ such that 
$$
\inf_{u \in \mathscr{A}_D} E_\eps(u) = \lim_{n\to \infty} E_\eps(u_n).
$$
We want to use the Banach-Alaoglu theorem which states that in every 
reflexive Banach space, every bounded sequence admits a weakly convergent subsequence.
First, $\mathscr{D}(\nabla_{\min})$ is indeed a reflexive Banach space: 
the map $u \mapsto (u, \nabla u)$ identifies $\mathscr{D}(\nabla_{\min})$
with a closed subspace of $L^2(M) \times L^2(M,TM)$; the latter is 
a reflexive Banach space and thus any closed subspace, such as $\mathscr{D}(\nabla_{\min})$, is a reflexive 
Banach space as well. 

Second, $(u_n) \subset \mathscr{D}(\nabla_{\min})$ is indeed a bounded 
sequence: $E_\eps(u) \geq 0$ for any $u \in \mathscr{D}(\nabla_{\min}) \cap L^4(M)$ and hence
the infimum of $E_\eps(u)$ over $u \in \mathscr{A}_D$ is finite; therefore $(E_\eps(u_n))_n$ has finite limit and in particular
is uniformly bounded. This shows boundedness of $(u_n) \subset \mathscr{D}(\nabla_{\min})$
in the graph norm, where we have used that the $L^2$-norm is bounded by the $L^4$-norm on manifolds 
of finite volume. 

Thus by the Banach-Alaoglu theorem, there exists a subsequence 
$(u_{n_k}) \subset \mathscr{D}(\nabla_{\min})$, which is weakly convergent, i.e. 
there exists $u_* \in \mathscr{D}(\nabla_{\min})$ such that for any $\phi \in L^2(M)$ and $\psi \in L^2(M,TM)$
\begin{equation}\begin{split}
\int_M u_{n_k} \phi \xrightarrow{k\to \infty} \int_M u_* \phi \quad 
\int_M g\bigl( \nabla u_{n_k}; \psi \bigr) \xrightarrow{k\to \infty} \int_M g\bigl( \nabla u_*; \psi \bigr).
\end{split}
\end{equation}
Since $(u^2_{n_k}-1) \subset L^2(M)$ is a bounded sequence in a reflexive 
Banach space as well, we can apply Banach-Alaoglu again and assume by passing to a 
subsequence, that 
$$
\int_M (u^2_{n_k}-1) \phi \xrightarrow{k\to \infty} \int_M (u_*^2-1) \phi.
$$
Taking the test function $\phi \equiv 1$, we find that weak convergence preserves the mass
constraint \eqref{mass}, i.e. $\spr{u} = m$. Since $L^2$-norms are weakly lower semi-continuous, we conclude 
\begin{equation}\label{lsc}\begin{split}
&\| \nabla u_*\|_{L^2} \leq \liminf_{k\to \infty} \| \nabla u_{n_k} \|_{L^2}, \\
&\| u_*^2 - 1\|_{L^2} \leq \liminf_{k\to \infty} \| u^2_{n_k} - 1 \|_{L^2}.
\end{split}
\end{equation}
We conclude that $u_* \in \mathscr{A}_D$ is indeed a minimizer, since by \eqref{lsc}
$$
E_\eps(u_*) = \frac{\eps}{4 \sigma} \| \nabla u_*\|_{L^2} +  \frac{1}{4} \| u_*^2-1\|_{L^2} 
\leq \liminf_{k\to \infty} E_\eps(u_{n_k}) = \inf_{u \in \mathscr{A}_D} E_\eps(u).
$$
\end{proof}

The same result holds if we impose additional, e.g.~generalized Neumann boundary
conditions in the definition of $\mathscr{A}_N$, since the trace operator $\textup{tr}$ in \eqref{trace-mapping}
is continuous in the operator norms and in particular weakly continuous. Thus, weak convergence 
preserves boundary conditions.

The existence in Proposition \ref{existence-prop} is sufficient to 
proceed with the convergence of $u_\eps$, and also to justify 
the numerics in \S\ref{numsec}. Thus, rather for completeness 
our next result shows that minimizers are in fact strong solutions of the Allen-Cahn equation. 
We do not proceed as in Evans \cite[\S 8]{evans}, but rather present an approach 
adapted to the present possibly singular setting. Consider the Lagrangian $L(u,\lambda)$ as in \eqref{lagrangian}
and note that a minimizer $u$ of $L(u,\lambda)$ satisfies
\begin{align}\label{weak-L}
\left. \frac{d}{ds} \right|_{s=0} L(u+s\phi, \lambda) \equiv 
\eps^2 \int_M g(\nabla u, \nabla \phi) + \int_M W'(u) \phi - \int_M \lambda \phi = 0,
\end{align}
where we vary among $\phi \in \mathscr{D}(\nabla_{\min})$ if $u \in \mathscr{A}_D$, and 
$\phi \in \mathscr{D}(\nabla_{\max})$ if $u \in \mathscr{A}_N$.
In other words, $u$ in $\mathscr{A}_D$ or $\mathscr{A}_N$ is a weak solution to the Allen-Cahn equation \eqref{ch2}, 
with the test functions $\phi$ lying in $\mathscr{D}(\nabla_{\min})$ or $\mathscr{D}(\nabla_{\max})$, respectively. 
Note that the intersection with $L^4(M)$ is no longer needed for the test functions. 

\begin{Prop}\label{regularity-prop} Consider any $u \in H^1(M) \cap L^4(M)$.
\begin{enumerate}
\item Let $u$ be a weak solution to the Allen-Cahn equation \eqref{ch2}, 
with test functions $\phi \in \mathscr{D}(\ \nabla_{\min}) \subseteq H^1_0(M)$. Then $u \in \mathscr{D}(\nabla_{\min}^* \nabla_{\min})$
and in particular, $\textup{tr} \, u = 0$.
\item Let $u$ be a weak solution to the Allen-Cahn equation \eqref{ch2}, 
with test functions $\phi \in \mathscr{D}(\nabla_{\max}) \equiv H^1(M)$. Then $u \in \mathscr{D}(\nabla_{\max}^* \nabla_{\max})$
and in particular, $\textup{tr} \circ \partial_\nu \, u = 0$.
\end{enumerate}
\end{Prop}

\begin{proof} We shall prove the first statement, the second one being verbatim with $\nabla_{\min}$ and $\Delta_D$
replaced by $\nabla_{\max}$ and $\Delta_N$, respectively.
We may rewrite the Allen-Cahn equation \eqref{ch2} as 
$\Delta u = \eps^{-2} \Big( W'(u) -\lambda \Big)$. Hence $u$ is a weak stationary (i.e. time independent)
solution to the inhomogeneous heat equation
\begin{align}\label{heat-equation}
(\partial_t - \Delta) \, \omega = - \eps^{-2} \Big( W'(u) -\lambda \Big).
\end{align}
Consider the self-adjoint extension $\Delta_D = - \nabla_{\min}^* \nabla_{\min}$
of the Laplace Beltrami operator on $(M,g)$, with domain $\mathscr{D}(\nabla_{\min}^* \nabla_{\min})$.
Consider the heat semigroup $e^{t\Delta_D}$ generated by $\Delta_D$.
A solution to \eqref{heat-equation} with initial condition $\omega(0)=u$ is given 
in terms of the heat semigroup by
\begin{align}\label{omega-def}
\omega := e^{t\Delta_D} u - \eps^{-2} e^{t\Delta_D} * \Big(W'(u) -\lambda\Big),
\end{align}
where $*$ indicates convolution in time. The proof now proceeds by studying regularity of $\omega$
and then proving $\omega \equiv u$. Since $u, W'(u) \in L^2(M)$ and the domain of $\Delta_D$ is dense in $L^2(M)$, $\omega$ is a mild solution to 
\eqref{heat-equation} in the sense of \cite[Definition 4.1.4]{Lun95}. By \cite[Theorem 4.3.1]{Lun95}, we conclude
that $\omega$ is in fact a classical solution, i.e. for any $T>0$ we have the regularity
\begin{align}\label{regularity-classical}
\omega \in C \Big( (0,T], \mathscr{D}(\Delta_D)\Big) \cap  C^1 \Big( (0,T], L^2(M) \Big).
\end{align}
Since $\mathscr{D}(\Delta_D) = \mathscr{D}(\nabla_{\min}^* \nabla_{\min})$, we can integrate by 
parts for any $\phi \in \mathscr{D}(\nabla_{\min})$
\begin{align}\label{by-parts}
  \int_M  g\big(\Delta_D \, \omega, \phi\big)  = - \int_M  g\big(\nabla \omega, \nabla \phi\big).
\end{align}
It remains to show $\omega \equiv u$. By construction, $(\omega - u)$ solves the following equation
$$
(\partial_t - \Delta) (\omega - u) = 0, \quad \omega(0) = u,
$$ 
weakly, i.e.  for any $\phi \in \mathscr{D}(\nabla_{\min})$ we have (cf. \eqref{weak-L} and \eqref{by-parts})
\begin{equation}\label{weak-w}
\int_M \big(\partial_t \, \omega\big) \phi + \int_M  g\big(\nabla (\omega - u), \nabla \phi\big) = 0.
\end{equation}
From here we compute by plugging $\phi = (\omega - u)$ into \eqref{weak-w}
\begin{equation}\begin{split}\label{decrease}
\frac{d}{dt} \Big\| \, \omega - u \, \Big\|^2_{L^2} &= 2 \int_M \big(\partial_t \, \omega\big) (\omega - u)
\\ &= - \int_M  g\Big(\nabla (\omega - u), \nabla (\omega - u) \Big) 
= - \Big\| \nabla (\omega - u) \Big\|^2_{L^2} \leq 0.
\end{split}\end{equation}
Thus, if $\|\omega(t) - u\|_{L^2}$ is continuous as $t \to 0$, then 
\eqref{decrease} together with $\omega(0)=u$ implies that $\omega(t) \equiv u$. From here the
statement follows with \eqref{regularity-classical}. Hence it remains to establish
continuity of $\|\omega(t) - u\|_{L^2}$ at $t=0$. Note first
\begin{align*}
\Big\| e^{t\Delta_D} * \Big(W'(u) -\lambda\Big) \Big\|_{L^2} \leq \int_0^t \Big\| e^{(t-\widetilde{t})\Delta_D} * 
\Big(W'(u) -\lambda\Big) \Big\|_{L^2} d\widetilde{t} \xrightarrow{t\to 0} 0.
\end{align*}
By the Lumer-Phillips theorem \cite[Theorem 3.1]{Lumer}, the heat operator $e^{t\Delta_D}$
is a strongly continuous semigroup, generated by $\Delta_D$. Indeed, $\Delta_D \leq 0$ is dissipative 
\cite[(1.1)]{Lumer} and the image of $(\textup{Id}-\Delta_D)$ is $L^2(M)$, since $1$ lies in the resolvent set of 
the closed operator $\Delta_D$. Hence the conditions of the Lumer-Phillips theorem are satisfied and 
by strong continuity of $e^{t\Delta_D}$ 
\begin{align*}
\Big\| e^{t\Delta_D} u - u \Big\|_{L^2} \xrightarrow{t\to 0} 0.
\end{align*}
In view of \eqref{omega-def}, we conclude that $\|\omega(t) - u\|_{L^2} \to 0$
as $t \to 0$ and hence by \eqref{decrease} we find $\omega(t) \equiv u$.
The statement now follows from \eqref{regularity-classical}. 
\end{proof}

\subsection{Convergence}
We want to extend the Modica-Mortola Theorem in 
\cite[Chapter 13.2]{Ri}, stated and proved for domains in $\R^d$, 
to compact manifolds $\Mbar$ of dimension $d$ with boundary and conical singularities. 

Let $W:\R\rightarrow [0,+\infty)$ be a continuous double-well potential with exactly two minima at $\pm 1$, e.g.,  $W(x) = \frac{1}{4} (x^2-1)^2$. Recall from \eqref{CH-intro-subsection} that 
$\sigma = \int_{-1}^1 \sqrt{W(s)/2} \, ds$. Let $\abs{M}_g$ denote the finite volume of $(M,g)$ and 
fix any mass 
%$m\in (-\abs{M}_g,\abs{M}_g)$
 $m\in(-1,1)$. The energy functional $E_\eps(u)$ in \eqref{eemin} is
defined a priori only for $u \in H^1(M) \cap L^4(M)$. We extend its definition to any $u\in L^1(M)$ 
by a simple trick (recall $\mathscr{A}_N$ is defined in Proposition \ref{existence-prop})
\begin{equation}\label{eq:energy functional}
\cE_\eps[u]:=\left\{\begin{array}{ll} \displaystyle E_\eps(u), & \text{if } u\in \mathscr{A}_N, \\
+\infty, & \text{if } u\in L^1(M) \backslash \mathscr{A}_N. \end{array}\right. 
\end{equation} 
This extension does not affect the minimizers: indeed 
the minimizers $u_N^\eps$ of $E_\eps(u), u \in \mathscr{A}_N$ are precisely the minimizers of 
$\cE_\eps[u], u \in L^1(M)$.

Consider the space $BV(M;\{-1,1\})$ of functions $u:M \to \{-1,1\}$ of 
bounded variation and recall
the notation $\langle u \rangle$ in the mass constraint \eqref{mass}. Then we set using the 
notion of perimeter $P_g$ of Caccioppoli sets in Definition \ref{def:CacciSets}
\begin{equation}\label{E_0def}
    \cE_0[u]:=\left\{\begin{array}{ll} P_g(\{x\in M:u(x)=-1\}) & \text{if } u\in BV(M;\{-1,1\}),  \langle u \rangle = m, \\
         +\infty, & \text{otherwise.}\end{array}\right.
\end{equation}
This measures the size of the boundary of $\{x\in M:u(x)=-1\}$, not including the part of it contained in  $\partial M$. While $BV(M) \subset L^1_{loc}(M)$, 
restricting the values to $\{\pm 1\}$ gives $BV(M;\{-1,1\}) \subset L^1(M)$ on manifolds $(M,g)$ of finite volume. 

Note also that the minimizers of $\cE_0[u], u \in L^1(M)$, are precisely those functions with values $\pm 1$ and satisfying $\langle u \rangle = m$ with a jump along 
hypersurfaces of minimal perimeter. We begin with a preliminary approximation result, which is where we have to be careful
about the conical singularities.  

\begin{Lem}\label{smoothing-lem}
Consider $u\in BV(M,\{-1,1\})$ with $\langle u \rangle = m$, such that $\cE_0[u]< \infty$.
Set $E=\{x\in M: u(x)=-1\}$. Then there exists a sequence of subsets $(E_n)_n \subset M$
with smooth boundaries $\partial E_n \subset M\backslash \partial M$ not intersecting the conical singularities of $M$, such that for 
$u_n := -\chi_{E_n}+\chi_{M\setminus E_n}$ the following holds ($d= \dim M$)
\begin{align}\label{smoothing}
\cE_0[u_n] \to \cE_0[u], \quad \cH^{d-1}(\partial E_n \cap \partial M)=0, 
\quad |(E \backslash E_n) \cup (E_n \backslash E)| \to 0.
\end{align}
\end{Lem}
\noindent
Here $|E|$ denotes the volume of the set $E$.
\begin{proof}
Applying \cite[Lemma 13.8]{Ri} locally in each coordinate neighborhood, 
     we can approximate $u$ by $v_n=-\chi_{F_n}+\chi_{M\setminus F_n}$, where $(F_n)$ is a sequence of subsets in $M$ of finite perimeter
     with boundary that is smooth in the interior of $M$, and 
\begin{equation*}
\cE_0[v_n] \to \cE_0[u], \quad |(E \backslash F_n) \cup (F_n \backslash E)| \to 0.
\end{equation*}
     Since each $F_n$ is smooth, and hence $F_n$ and and $M\setminus F_n$ each contain a non-empty open ball, we can apply \cite[Lemma 1]{Mo} or \cite[Lemma 13.7]{Ri}
     locally in each coordinate chart, to find another approximation by $w_n=-\chi_{G_n}+\chi_{M\setminus G_n}$, where the subsets $(G_n)$ are of finite perimeter, 
     with boundary smooth in the interior of $M$, and 
$$
\cE_0[w_n] \to \cE_0[u], \quad \cH^{d-1}(\partial G_n \cap \partial M)=0, 
\quad |(E \backslash G_n) \cup (G_n \backslash E)| \to 0.
$$ 
The interfaces $\partial G_n$ may however intersect the conical singularities of $M$. 
In order to avoid that, let us assume first that $M$ has a single conical singularity 
$\mathscr{C}(N)=(0,1) \times N$. Denote the radial coordinate  by $x\in (0,1)$, and extend it as a smooth function to all of $M$, with $x\geq1$ outside $\mathscr{C}(N)$.
Sard's theorem applied to the function $x\restriction \partial G_n$ on $\partial G_n$ implies that there exists $\eps_n \in (0,1/n)$
such that $V_n:= \partial G_n \cap \{x= \eps_n\}$ is a smooth submanifold of $G_n$, for each $n\in\N$. By construction $E'_n := G_n \cap \{x\geq \eps_n\}$ has boundary $\partial E'_n = G'_n\cup G''_n$ where
\begin{align*}
G'_n &:= G_n\cap \{x=\eps_n\}, \\
G''_n&:= \partial G_n \cap \{x\geq \eps_n\},
\end{align*}
and the submanifolds $G'_n$ and $G''_n$ intersect each other in their common boundary $V_n$ transversally. The sequence
$E'_n$ satisfies \eqref{smoothing}.
But $\partial E'_n$ is not smooth at the corner $V_n$. 
This setting is illustrated in Figure \ref{smoothing1-fig}.

 \begin{figure}[h]
\begin{center}
  \includegraphics[width=0.8\linewidth]{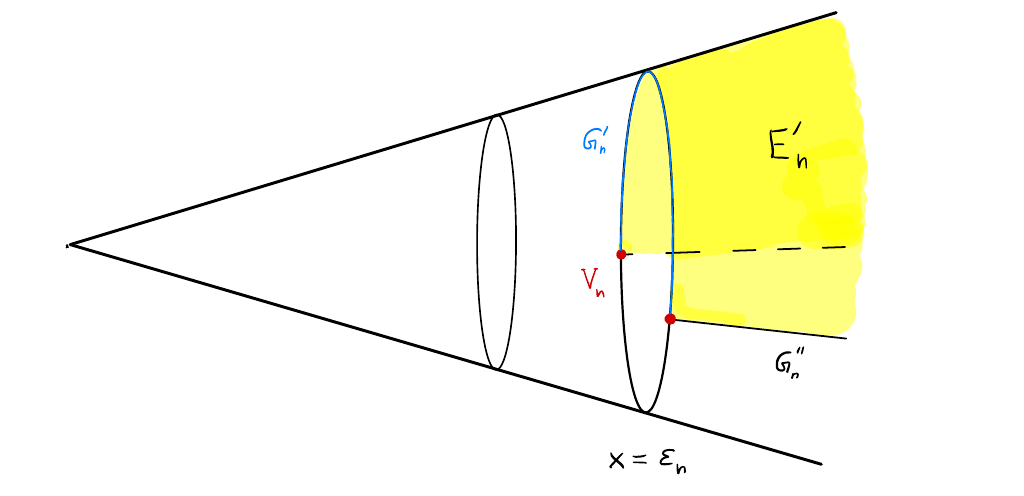}
  \end{center}
  \caption{Illustration of $V_n, E'_n, G'_n$ and $G''_n$.} \label{smoothing1-fig}
\end{figure}

However, we may smoothen out the corner as follows. Write $G'_n = \{\eps_n\}\times \Gtilde'_n$ and $V_n = \{\eps_n\}\times \Vtilde_n$ with $\Gtilde'_n,\Vtilde_n\subset N$.  Denote by $y_n:\Gtilde'_n\to[0,\infty)$ a boundary
defining function of $\Vtilde_n$ in $\Gtilde'_n$, i.e.\ $\Vtilde_n=\{y_n=0\}$ and $\nabla y_n\neq0$ at $\Vtilde_n$, and chosen to have $1$ in its range as a regular value. Consider a smooth function 
$\gamma: (\eps_n/2, \eps_n)_x \to  (0, 1)_{y_n}$ whose graph forms a smooth curve together
with the half-lines $(x\geq \eps_n, y=0)$ and $(x=\eps_n/2, y\geq 1)$. This is illustrated in 
Figure \ref{smoothing2-fig}.
 
 \begin{figure}[h]
\begin{center}
  \includegraphics[width=0.5\linewidth]{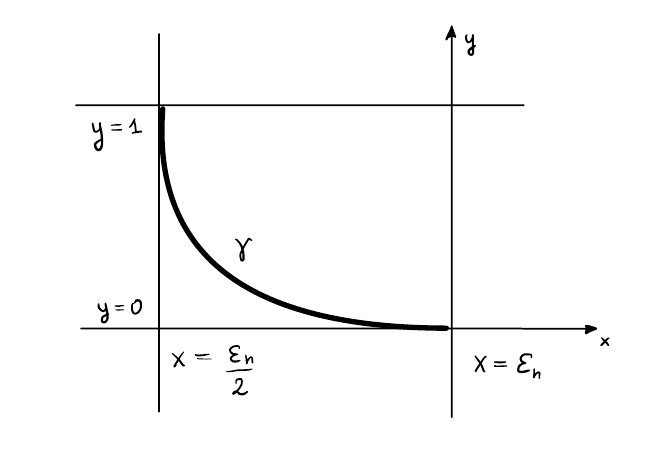}
  \end{center}
  \caption{Illustration of curve $\gamma$.} \label{smoothing2-fig}
\end{figure}

We can now smoothen out the corner in $E'_n$ by setting as in Figure \ref{smoothing3-fig}
$$
E_n := E'_n \cup \{(x,q) \in (\eps_n/2,\eps_n) \times \Gtilde'_n\,: \ \ y_n(q) > \gamma(x) \}\,.
$$

 \begin{figure}[h]
\begin{center}
  \includegraphics[width=0.8\linewidth]{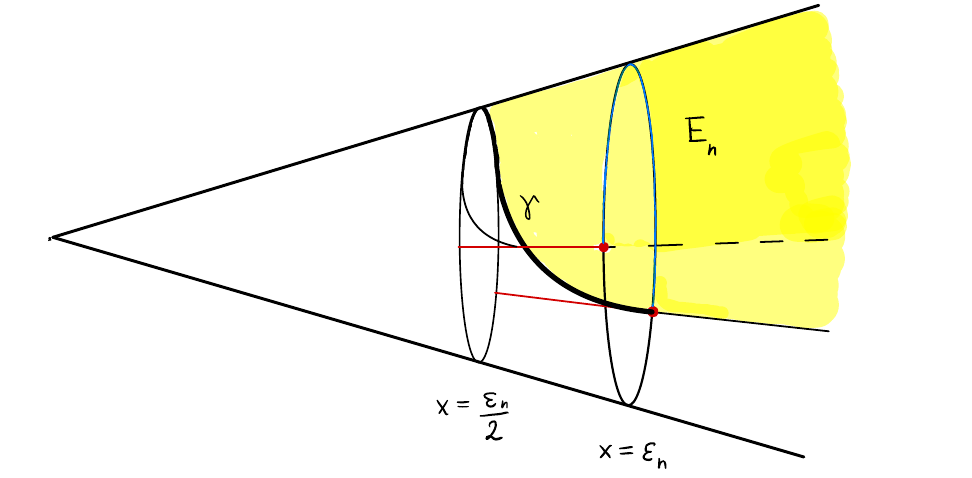}
  \end{center}
  \caption{Illustration of $E_n$.} \label{smoothing3-fig}
\end{figure}

Because $E_n$ differs from $G_n$ only in a small conical neighborhood $(0,\eps_n) \times N$
with $\eps_n \to 0$ as $n \to \infty$, $(E_n)$ satisfies \eqref{smoothing}, and this proves the statement.
In the general case of finitely many conical singularities, we repeat such procedure in each conical neighborhood. 
\end{proof}

We can now prove our main result here, namely the $\Gamma$-convergence of the functionals above. 

\begin{theorem}\label{thm:1}
$\cE_\eps$ $\Gamma$-converges to $\cE_0$ as $\eps\rightarrow0^+$ with respect to the strong $L^1$-topology.        
\end{theorem}

\begin{proof} 
To prove $\Gamma$-convergence we need to prove the $\liminf$-inequality (\ref{eq:liminf-ineq}) and 
additionally, by Remark \ref{rem:limsup-ineq}, construct a sequence that fulfills the $\limsup$-inequality \eqref{eq:limsup-ineq}. 
The underlying complete metric space is $L^1(M)$. 

We shall start with the $\liminf$-inequality. Let $u_\eps\longrightarrow u$ in $L^1(M)$ as $\eps\rightarrow 0$ 
and assume $\displaystyle\liminf_{\eps\rightarrow 0}\cE_\eps[u_\eps]<\infty$ (otherwise there is nothing to prove), 
i.e. we may assume without loss of generality that $u_\eps\in  \mathscr{A}_N$ for all $\eps>0$. Consider a subsequence 
$\eps_n\longrightarrow 0$ such that $u_{\eps_n} \to u$ pointwise almost everywhere as $n\rightarrow\infty$. 
Then by Fatou's lemma we obtain
 \begin{align*}
     0\leq \int_M W(u) &= \int_M \liminf_{n\rightarrow \infty} W(u_{\eps_n})
     \leq \liminf_{n\rightarrow \infty} \int_M W(u_{\eps_n}) 
     \\ &\leq \liminf_{n\rightarrow \infty} \int_M \frac{\eps_n^2}{2} \,
     |\nabla_g u_{\eps_n}|^2+ W(u_{\eps_n}) = \liminf_{n\rightarrow \infty} 2\sigma \eps_n\cE_{\eps_n}[u_{\eps_n}]=0,
 \end{align*}
 because $\displaystyle\liminf_{n\rightarrow\infty}\cE_{\eps_n}[u_{\eps_n}]<\infty$. Hence, $W(u) = 0$ and thus also
$u$ takes values in $\{-1,1\}$ almost everywhere. Furthermore, we compute
 \begin{align*}
     \liminf_{n\rightarrow\infty}&\, \cE_{\eps_n}[u_{\eps_n}]=\liminf_{n\rightarrow\infty}  \frac{1}{2\sigma}\int_M 
     \frac{\eps_n}{2} |\nabla_g u_{\eps_n}|^2+\frac{1}{\eps_n}W(u_{\eps_n}) 
     \\&\geq \liminf_{n\rightarrow\infty}\int_M \frac{1}{\sqrt{2}\sigma} |\nabla_gu_{\eps_n}|\sqrt{W(u_{\eps_n})} 
     \geq \liminf_{n\rightarrow\infty}\int_M  \frac{1}{\sqrt{2}\sigma} |\nabla_g(h\circ u_{\eps_n})|\\& \geq 
      \frac{1}{\sqrt{2}\sigma}\norm{\nabla_g(h\circ u)}_g(M)
     = \frac{(h(1)-h(-1))}{\sqrt{2}\sigma} P_g(\{x\in M: u(x)=-1\}) \\
     &=P_g(\{x\in M: u(x)=-1\})  \equiv \cE_0[u],
 \end{align*}
 where $h(t):=\int_0^t\sqrt{W(s)} \, ds$ 
and we used \eqref{eq:lower semi-cont} in the last inequality step.
This proves the $\liminf$-inequality.
 
 In the second step we have to construct a recovery sequence $(u_\eps)_\eps$ that converges to $u$ in $L^1(M)$ and fulfills the $\limsup$-inequality
 \eqref{eq:limsup-ineq}. If $\cE_0[u]=\infty$, there is nothing to prove and hence we assume without loss of generality that $\cE_0[u]< \infty$ and thus 
 $u\in BV(M,\{-1,1\})$ and $\langle u \rangle = m$. 
 
 Note that we do not assume that $u$ is a minimizer of $\cE_0$, hence the 
interface $\partial E$ need not be smooth, even for a compact smooth $M$.
 In view of Lemma \ref{smoothing-lem} however, we may assume without loss 
 of generality that in our quest for a recovery sequence, $u\in BV(M,\{-1,1\})$ has the property that $E$ 
 is of finite perimeter, has smooth boundary disjoint from conical singularities and satisfies $\cH^{d-1}(\partial E\cap \partial M)=0$. 
 Also, we may assume that $E$ is open. Since $\partial E$ may be assumed to be disjoint from 
 the conical singularities, the remainder of the argument is performed as in the non-singular case. 
 
 The construction of the recovery sequence $(u_\eps)_\eps$ is 
nowadays classical, 
 see \cite[Prop. 2 p. 133]{Mo} and also \cite[Theorem 13.6, p. 383-386]{Ri} for open subsets of $\R^n$. Instead of repeating the steps therein, 
 we refer to the Riemannian version e.g. in \cite[Proposition 3.3]{BNAP22}, which constructs $(u_\eps)_\eps$ such that 
     \begin{align*}
     & \displaystyle\limsup_{\eps\rightarrow 0^+}\cE_\eps[u_\eps]\leq
     P_g(E,M) \equiv \cE_0[u].
     \end{align*}
     The construction is local near the smooth boundary $\partial E$. Since 
     by assumption $\partial E$ is disjoint from the conical singularity, the arguments remain 
     unchanged in our setting. By Remark \ref{rem:limsup-ineq} 
     $\Gamma$-convergence now follows.
\end{proof}
 
 A fundamental consequence of $\Gamma$-convergence is convergence of minimizers and 
 the abstract result \cite[Proposition 4]{Mo} together with Theorem \ref{thm:1} implies the following
 
 \begin{Prop}
For every $\eps>0$, let $u_\eps \in L^1(M)$ be a minimizer of $\cE_\eps$. 
If $u_\eps \to u$ in $L^1(M)$ as $\eps\rightarrow 0^+$, then $u$ is a minimizer of $\cE_0$ in $L^1(M)$ and 
$\displaystyle\lim_{\eps\rightarrow0^+}\cE_\eps(u_\eps)=\cE_0(u)$.
 \end{Prop}

\section{Minimizers on cones in 2D: interfaces}\label{Danielsec}
To study the limits at $\eps=0$ of minimizers of \reff{ch2} in a singular setting, we 
use truncated cones of height $h$ with elliptic base at $z=0$ 
of semi axes $1$ and $a\ge 1$, parameterized 
over the unit disk, i.e.
\huga{\label{cpara}
\bpm \xti\\\yti\\\zti\epm=\phi(x,y):=\bpm ax\\
y\\h(1-(x^2+y^2)^{1/2})\epm,\quad (x,y)\in\Om
=\{x^2{+}y^2\le 1\}. 
}
In view of the forthcoming results in \S\ref{covsec}--\S\ref{ncsec4}, we first discuss the three 
types of limit interfaces which we expect from 
the types  T1 (tip), T2 (winding),
and T3 (horizontal) from Fig.\ref{f1} for $\eps\to 0$.  
They divide the area into two equal halves since $m=0$, and by the general theory they are 
critical points of the length functional, so they have constant curvature (except possibly where they pass through the conical singularity), and 
we call their respective lengths $l_1$, $l_2$ and $l_3$. \medskip

We have two results 
on the question which interface type is a length minimizer. The first is 
semi--analytical, giving a rather explicit computation of the 
interface lengths at $\eps=0$ on a circular cone. 
See also Fig.~\ref{f6}(b$_1$) for a comparison with $E_\eps$ for the 
respective solutions at small $\eps>0$. 

\begin{figure}[h]
\bce
\btab{ll}{{\sm (a)}&{\sm (b)}\\
\ig[width=0.27\tew]{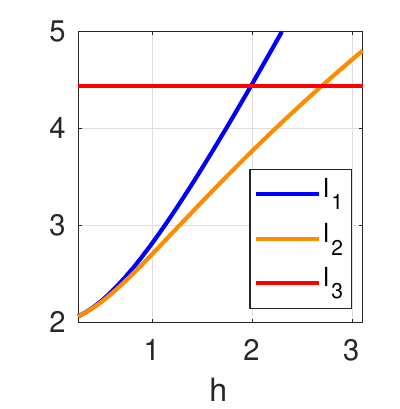}&
\rb{4mm}{\ig[width=0.36\tew]{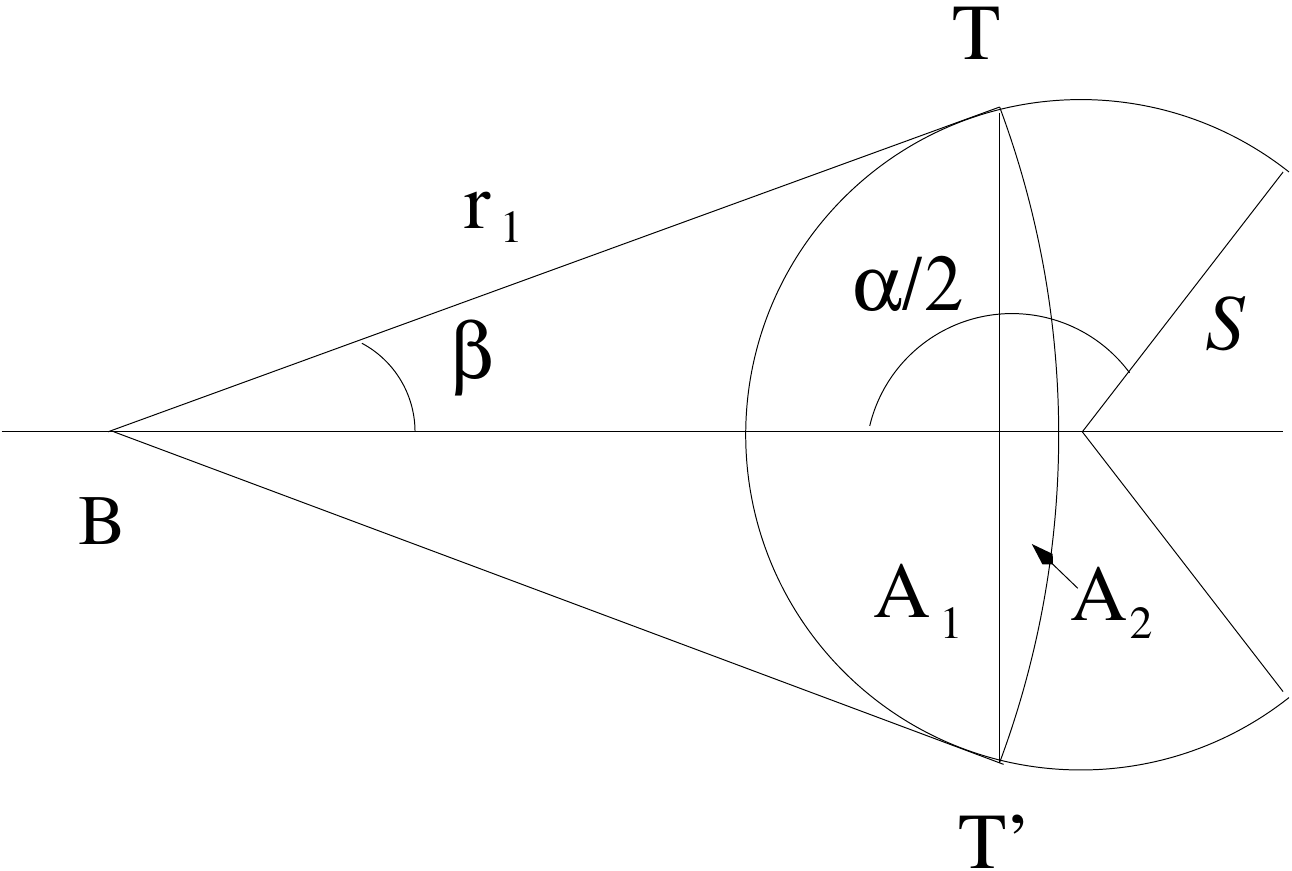}\hs{4mm}}
}
\ece\vs{-3mm}
\caption{{\sm  (a) $l_j(h)$. (b) Sketch for computing $l_2$. The pacman shape arises from cutting open the cone along a radius, and flattening it into the plane.
\label{f7}}}
\end{figure}

\begin{resu}\label{resu1}
 On a circular cone ($a=1$ in \eqref{cpara}) the lengths $l_j$ behave as in Figure \ref{f7}(a) as a function of the height $h$ of the cone. In particular, the tip interface is never a minimizer, and the circular interface is a minimizer for large $h$. 
\end{resu}
\begin{proof}
The slant height of the circular cone of height $h$ is $S=\sqrt{1+h^2}$, so $$l_1(h)=2\sqrt{1+h^2}\,.$$
The cone  can be explicitly 
flattened to a circular sector $C$ of radius $S$
 and angle 
$\al=2\pi/\sqrt{1+h^2}$, see Fig.~\ref{f7}(b), and hence area 
$A=\al S^2/2=2\pi\sqrt{1+h^2}$. Then, 
$$l_3=\sqrt{2}\pi\,: \text{ circular interface at height $\tilde h=(1-1/\sqrt{2})h$}.$$
Finally, $l_2(h)$ can  be computed semi--analytically as follows: 
First we seek the point $T$ such that the circular arc from $T$ to its reflection $T'$ divides 
$C$ into two equal areas $\al S^2/4$. For given $T$, the intersection of 
the tangent to the circle at $T$ with the $x$--axis is at 
$x_B=-S/\sin(\beta)$ relative to the center, and the areas $A_1(\beta)$ and $A_2(\beta)$ 
of the two circular segments to the right and left of the line $\overline{TT'}$ 
are given by  
\huga{
\text{$A_1(\beta)=\frac{r_1^2}{2}(2\beta{-}\sin(2\beta))$, 
and 
$A_2(\beta)=\frac{S^2}{2}(\pi{-}2\beta{-}\sin(2\beta))$,  
}}
where  $r_1=S\cot\beta$. 
Thus, we first numerically solve 
$A_1(\beta)+A_2(\beta)=\al S^2/4$ 
for $\beta\in(0,\pi/2)$, and then compute the length $l_2(h)=2\beta r_1$. 
\end{proof}

The next result applies to surfaces with general conical singularities 
of angle $\al<2\pi$, which for instance applies to the conic metric 
$dr^2+c^2 r^2 d\theta^2$ on 
$(0,1)\times S^1$ with angle $\al=2\pi c$ if and only if $ c<1$, and 
to the elliptic cones \eqref{cpara} with general $a\ge 1$, where 
$\al<2\pi$ if $h>0$, while naturally $\al=2\pi$ if $h=0$.  
Conical surfaces with angles bigger than $2\pi$ can be constructed as follows, 
see also Appendix \ref{coneHD} for further discussion. 
If we choose a simple closed curve in the unit sphere in $\R^3$ and 
take $M$ to be the union of segments $0p$ for $p$ on the curve then we obtain a conical surface with the angle $\al$ being the curve length. 
If $\al\ge 2\pi$, then minimizers may in general pass through the 
singularity (as is the case for a flat disk, which has $\al=2\pi$, where any diameter is a minimizer for $m=0$), but we now show that tip interfaces are never minimizers  
for $\al<2\pi$. 

\begin{Prop} \label{prop:tip interface}
Let $\Mbar$ be a surface with a conical singularity $P$ of angle $\al<2\pi$. 
Then any curve of minimal length which divides $\Mbar$ into two parts of prescribed  area ratio does not pass through $P$. 
\end{Prop}
Prescribing an area ratio corresponds to choosing a general mass parameter $m$ with $|m|<1$ in \eqref{ch2}(b). The case $m=0$ corresponds to equal areas. (For elliptic cones it was natural to consider only $m=0$ because of reflection symmetry.)

\begin{figure}[h]
\bce
\ig[width=0.3\tew]{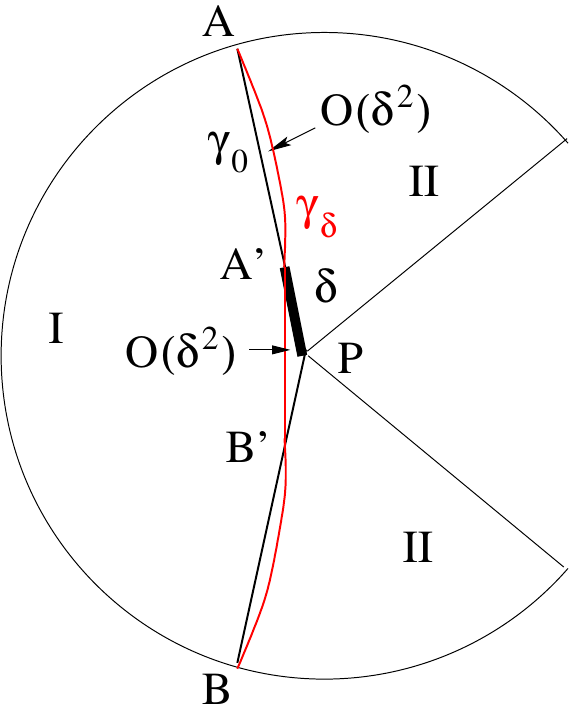}
\ece\vs{-3mm}
\caption{{\sm Modifying a tip interface $\ga_0$ to a shorter interface $\ga_\del$. 
\label{fgcf}}}
\end{figure}

\begin{proof}
Let $\gamma_0$ be a curve of minimal length satisfying the area constraint, and assume that it passes through the tip $P$. We  refer to Fig.~\ref{fgcf} for a sketch (compare Fig.~\ref{f7}(b)), 
and construct a shorter such curve $\ga_\del$.  
Consider a neighborhood $U$ of $P$ in $M$ of the form $(0,1)\times S^1$ (see Appendix \ref{coneHD}), on which $\gamma_0$ consists of a part $AP$ entering $P$ and a part $PB$ leaving $P$. Each part has constant curvature, since $\gamma_0$ is assumed to be of minimal length. In Fig. \ref{fgcf} 
they are depicted as straight lines for simplicity. 

Denote the  components of $U\setminus \gamma_0$ by $I$ and $II$; because the angle at $P$ is less than $2\pi$ at least one of them, say $I$, has an angle less than $\pi$ at $P$. 
For small $\delta>0$ let $A'$, $B'$ be points on $PA$, $PB$ at distance $\delta$ from $P$.

The main point is this: Replacing the part $A'PB'$ of $\gamma_0$ by the straight geodesic segment $A'B'$ reduces its length linearly, i.e.\ by $\geq c\delta$ where $c>0$.
On the other hand, the area of the triangle $A'PB'$ is only $O(\delta^2)$ since a $\delta$-neighborhood of $P$ has area $O(\delta^2)$, so the area of $I$ and $II$ is changed only by $O(\delta^2)$.

Therefore, we can replace the segments $AA'$ and $BB'$ by curves from $A$ to $A'$, and from $B$ to $B'$, respectively,
which bulge into $II$ by $O(\delta^2)$, so that the curve $\gamma_\delta$ 
obtained from this still satisfies the area constraint and has length $\ell(\gamma_0)-c\delta+O(\delta^2)$, which is less than $\ell(\gamma_0)$ for small $\delta>0$.
\end{proof}

\section{Minimizers on cones: Numerical illustrations}\label{numsec}
\subsection{The setup, and basic notions from bifurcation theory} \label{ncsec0}
To illustrate Theorem \ref{mthm1} and study 
the minimizers of \reff{ch2} for small $\eps>0$ in a singular setting we 
use the truncated cones of height $h$ and 
semi axes $1$ and $a\ge 1$ 
as in \eqref{cpara}.  
The metric determinant is $g=a^2+h^2(x^2+ay^2)/r^2$, 
and $dS=\sg\dd (x,y)$, and the corresponding
Laplace Beltrami operator is given by
\hual{
\Del u(x,y)=\sgi\biggl[&
\pa_x(\sgi (1+h^2y^2r^{-2})\pa_x u)-\pa_y(\sgi h^2xyr^{-2}\pa_x u)\notag\\
&-\pa_x(\sgi h^2xyr^{-2}\pa_y u)+\pa_y(\sgi(a^2+h^2x^2r^{-2})\pa_y u)
\biggr],\label{LBc}
}
where the coefficients in the divergence form are $L^\infty$,
but not $C^0$. 
We choose the standard double well potential
$W=\frac 1 4(u^2-1)^2$,  and the energy 
$$
E_\eps=\frac 1 {2\sig}\int_\Om \frac \eps 2 |\nabla u|^2+\frac 1 \eps W(u)\dd S, 
$$
where $\sig=\int_{-1}^1\sqrt{W/2}\dd u=\sqrt{2}/3$. To recall, our 
problem thus is 
\huga{\label{ch2b}
\barr{l}
{\rm (a)}\ \ G(u):=-\eps^2\Del u+W'(u)-\lam\stackrel!=0\text{ in }\Om, \quad \partial_\nu u=0\text{ at }\partial\Omega,\\
{\rm (b)}\ \  q(u):=\spr{u}-m\stackrel!=0. \earr
} 
We use the software \pdep\ \cite{p2p,p2phome} to discretize \reff{ch2b} 
by the Finite Element Method (FEM), and to treat the 
obtained algebraic system as a continuation and bifurcation problem. 
More specifically, we consider the weak form of (\ref{ch2b}a) and hence 
find critical points of $E_\eps$ in $H^1(M)$. 
For all $m$, \reff{ch2b} has the spatially homogeneous solution 
$u\equiv m$, $\lam=W'(m)$, i.e., we have the ``trivial branch'' $u\equiv m$, 
with $E_\eps(u)=\frac 1 {2\sig\eps} W(m)$. Starting at $m=1$ we 
first find branch points (BPs) from this trivial branch, and 
continue the bifurcating branches to $m=0$ at fixed $h,\eps>0$ and $a\ge 1$, 
 thus obtaining a 
selection of critical points $u$ of $E_\eps$ at $m=0$. 
Subsequently we 
continue some of these solutions $u_{\eps,h,a}$ in other parameters, 
including $\eps\to 0$, aiming to identify 
limits $u_0$ and the associated interfaces $I_0$, 
to check the formula $|I_0|=\lim_{\eps\to 0}E_\eps(u_{\eps,h,a})$, 
and to altogether identify minimal interfaces (depending on $h$ and $a$). 

By varying $a$ and $h$ we study 
the dependence of minimizers on the cone geometry; additionally, choosing 
$a\ne 1$ is useful to break symmetry: For $a=1$ we have a circular base, and 
altogether an $O(2)$ equivariant problem, which inter alia means that  
\bci 
\item all bifurcations from the spatially homogeneous branch with 
angular dependence are double, i.e.\ the linearization at the branch point has a 
two-dimensional kernel; 
\item to continue such branches we need a phase condition to uniquely 
choose solutions from the group orbit of rotations. 
\eci 
\pdep\ has methods to deal with these issues, but the situation 
is somewhat simpler and more generic if the rotational symmetry is broken. 
For elliptic cones, the BPs from the trivial branch are generically simple 
and hence we can apply the Crandall--Rabinowitz Theorem \cite{CR71} 
to obtain the bifurcating branches, while 
for circular cones we would need equivariant bifurcation theory 
(\cite[\S2.5]{p2p} and the references therein), and also the numerics 
would become slightly more involved.  
However, we still have three discrete symmetries: 
\hual{
\ga_1:\quad&(u(x,y),m,\lam)\mapsto(u(-x,y),m,\lam),\quad \Z_2\text{ symmetry 
over the $y$ axis;} \\
\ga_2:\quad&(u(x,y),m,\lam)\mapsto(u(x,-y),m,\lam),\quad \Z_2\text{ symmetry 
over the $x$ axis;}\\
\ga_3:\quad&(u(\cdot,\cdot),m,\lam)\mapsto(-u(\cdot,\cdot),-m,-\lam),\quad 
\text{ \btab{l}{$\Z_2$ symmetry in $u$, mass\\ and Lagrange multiplier.}}
}
These symmetries generate the symmetry group $\Ga$ of the 
problem via composition. 
$\ga_3$ can essentially be exploited to restrict to $m\le 0$, and 
the two mirror symmetries $\ga_1, \ga_2$ can be used to {\em a priori} 
determine whether bifurcations from simple BPs on the trivial branch 
are {\em transcritical} or {\em pitchforks}, see Remark \ref{bbrem}. 
\brem\label{brem}
{\rm For convenience, here we recall the basic examples (normal forms) 
for steady bifurcations, referring to \cite{p2p} and the references 
therein for further terminology and details, 
in particular \cite{kuz}.  Consider 
the following scalar bifurcation problems, with $\mu\in\R$ as a 
(generic) parameter. 

\bci 
\item[(a)] $f(u,\mu)=\mu-u^2=0$. This has the 
{\em solution branch} $u=\pm\sqrt{\mu}$ for $\mu\ge 0$, which shows a {\em fold} 
at $\mu=0$. This is also called {\em saddle-node} bifurcation as the 
lower branch $u=-\sqrt{\mu}$ contains saddles (unstable solutions) for the 
ODE $\dot u=f(u,\mu)$, while the upper branch contains stable nodes, 
see Fig.~\ref{f2}(a), which also shows the associated energy 
$E=-\mu+\frac 1 3 u^3$ such that $\dot u=-E'(u)$.  
\footnote{Any scalar ODE is a 
gradient system $\dot{u}=-E'(u)$, and if $E$ is bounded 
below and has only isolated critical points, then any solution must converge 
to a critical point of $E$.}

\item[(b)] $f(u,\mu)=\mu u+u^2$. For all $\mu\in\R$ we have the trivial solution 
$u=0$, and additionally the ``non--trivial branch'' $u=-\mu$. $u=0$ is stable 
(unstable) for $\mu<0$ ($\mu>0$), while $u=-\mu$ is stable (unstable) 
  for $\mu>0$ ($\mu<0$). Thus, at the {\em branching point} 
$(u,\mu)=(0,0)$ there is an {\em exchange of stability}. The 
bifurcating branch $u=-\mu$ exists on both sides of the critical value 
$\mu=0$ and hence the bifurcation is called {\em transcritical}, see Fig.~\ref{f2}(b). 

\item[(c)] $f(u,\mu)=\mu u-u^3$. As in (b), the trivial 
solution $u=0$ is locally unique, except at $\mu=0$, 
where  the non--trivial branches  $u=\pm\sqrt{\mu}$, $\mu>0$ bifurcate. 
Moreover, $u=0$ is stable (unstable) for $\mu<0$ ($\mu>0$), and 
$u=\pm\sqrt{\mu}$ are both stable. This is called a {\em supercritical 
pitchfork}, see Fig.~\ref{f2}(c), while for instance for $f(u,\mu)=\mu u+u^3$ we obtain a 
{\em subcritical} pitchfork with the bifurcating unstable 
branches $u=\pm\sqrt{-\mu}$. 
\eci 
\begin{figure}[ht]
  \bce 
\btab{p{48mm}p{48mm}p{48mm}}{
{\sm (a) fold}&{\sm (b) transcritical}&{\sm (c) pitchfork}\\
\btab{l}{\ig[width=44mm,height=36mm]{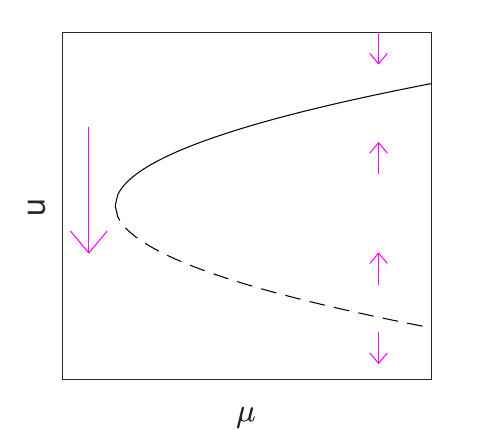}\\
\hs{-2mm}\ig[height=17mm]{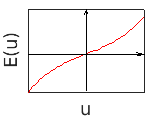}\hs{-1mm}\ig[height=17mm]{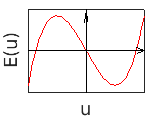}}
&\btab{l}{\ig[width=44mm,height=36mm]{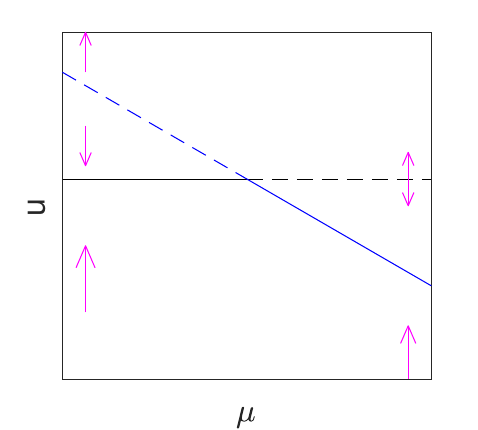}\\
\hs{-2mm}\ig[height=17mm]{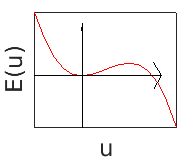}\hs{-1mm}\ig[height=17mm]{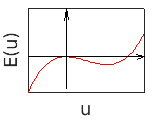}}
&\btab{l}{\ig[width=44mm,height=36mm]{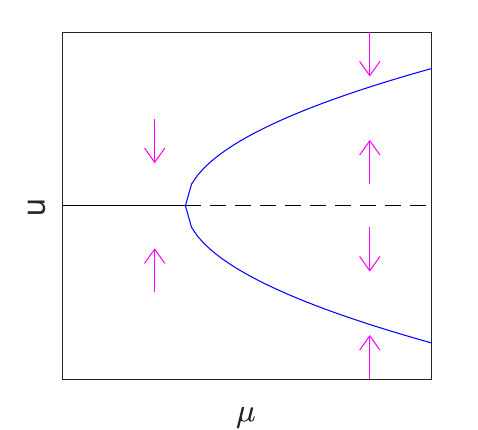}\\
\hs{-2mm}\ig[height=17mm]{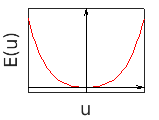}\hs{-1mm}\ig[height=17mm]{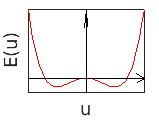}}}
\ece\vs{-4mm}
\caption{{\small (a-c) Elementary steady bifurcations. 
Full lines indicate branches of 
stable solutions and dashed lines indicate unstable branches, 
with the arrows 
indicating the flow of the associated ODE $\dot u=f(u,\mu)=-\pa_u E(u,\mu)$, 
with the energies indicated below the bifurcation diagrams at the 
respective left and right ends in $\mu$  
for each of the cases. 
}
  \label{f2}}
\end{figure}

(a),(b) and (c) (and further types of bifurcations) 
can occur along a single branch; for instance, $f(u,\mu){=}\mu u{+}u^3-u^5$ 
yields a subcritical pitchfork at $(u,\mu)=(0,0)$, with 
the nontrivial branches stabilizing in folds at 
$(u,\mu)=(\frac{\pm  1}{\sqrt{2}},\frac{-1}{4})$. 
The scalar normal forms typically 
arise as {\em bifurcation equations} 
via Liapunov--Schmidt reduction of, e.g., a PDE problem where along 
a solution branch a simple eigenvalue of the linearization 
goes through $0$ (bifurcation from simple eigenvalues),  
where a non-trivial kernel of the linearization is a necessary 
condition for (steady) bifurcation due to the implicit function theorem. 
Importantly, only the saddle--node bifurcation (a) is {\em generic}, 
while (b) and (c) require some special structure to occur generically, for 
instance the $\Z_2$ symmetry $u\mapsto -u$ in case of (c). Such 
symmetries often arise from physics, or in case of PDEs from the 
considered domain, again see \cite[\S2.5]{p2p} for further discussion, 
and \cite{GoS2002, hoyle}. Finally, note that we obtain the 
smooth plots in (a)--(c) because we consider 1D systems; in higher 
dimensions, for instance the ``shape of folds'' strongly depends 
on the chosen projection (the chosen functional on the ordinate) for 
plotting the BDs. See for instance Fig.\ref{f4}(a$_1$) vs Fig.\ref{f4}(d).}
\eex \erem 

\brem\label{bbrem}
{\rm 
In terms of the notions recalled in Remark \ref{brem}, the symmetries $\ga_1$ 
and $\ga_2$ have the following consequences for bifurcations of nontrivial 
branches from the trivial branch $u\equiv m$ in \reff{ch2b}: 
\bci 
\item If the solutions $u$ on the bifurcating branch satisfy exactly one of the conditions 
$\ga_1u=u$, $\ga_2u=u$ then the bifurcation must be a pitchfork, 
as together with $u$ also its partner $\ga_j u$ must bifurcate, $j=1,2$. This 
for instance applies to the branches yielding the tip (T1) and winding 
(T2) interfaces from Fig.~\ref{f1}. 
\item Solutions which violate both symmetries $\ga_1$ and $\ga_2$ 
cannot bifurcate from the homogeneous branch in a simple BP, i.e., 
they must arise in secondary bifurcations. (We do not further consider 
such solutions here). 
\item On the other hand, branches with solutions satisfying both symmetries 
$\ga_1u=u=\ga_2u$ (e.g., ``horizontal'' solutions like in Fig.~\ref{f1}(c), T3)  
generically bifurcate transcritically from the homogeneous branch. \eex
\eci 
}
\erem 

\brem\label{crem}
{\rm Besides the 1--parameter bifurcation problems from Remark \ref{brem}, 
in practice one often has to deal with $n\ge 2$--parameter problems, 
e.g., the 4--parameter $(m,\eps,h,a)$ problem \reff{ch2b}. In these, 
often one can fix $n{-}1$ parameters and let just one parameter vary, 
but there may be {\em co--dimension} $\ell$ points in parameter space, 
where one needs $\ell\ge 2$ parameters to capture the 
behavior of the system. One of the simplest examples is the so called 
{\em cusp catastrophe} \cite[\S8.2]{kuz}, with normal form 
$f(u,\mu_1,\mu_2)=-4u^3+2\mu_1 u-\mu_2$, see Fig.\ref{cfig}. For $\mu_1\le 0$ we 
always have exactly one stable solution  of $f(u,\mu_1,\mu_2)=0$ 
for all $\mu_2$, but for $\mu_1>0$ 
we have three solutions between the two fold--curves $\mu_2=\pm\mu_2(\mu_1)$, 
with the middle solution unstable and the two other solutions stable, 
and where the two fold curves form a cusp at the 
co--dimension two point $(\mu_1,\mu_2)=(0,0)$. 
To compute such cusps, it is often useful to compute the fold--curves 
by {\em fold point continuation} \cite[\S3.6.1]{p2p}, \cite{spctut}. 
See Fig.~\ref{f5} for an example for \reff{ch2b}. }\eex\erem

\begin{figure}[ht] \bce 
\ig[height=50mm]{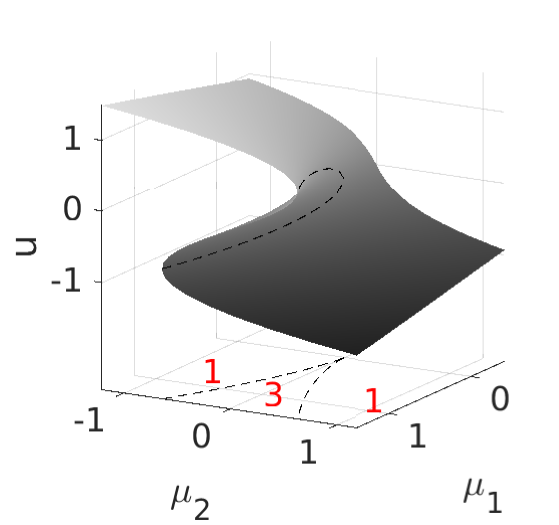}
\ece\vs{-4mm}
\caption{{\small The solution surface of $f(u,\mu_1,\mu_2):=-4u^3+2\mu_1 u-\mu_2=0$, and the projection $\mu_2=\pm\mu_2(\mu_1)$ of the fold curves on the 
$\mu_1$--$\mu_2$ plane, forming a cusp.}
 \label{cfig}}
\end{figure}

\subsection{Organization of the continuation in 
the different parameters} \label{covsec}
Given the remarks from \S\ref{ncsec0} and the results 
from \S\ref{Danielsec}, we organize the numerical 
continuation at $\eps>0$ as follows. 
\bci 
\item In Fig.\ref{f4} 
we compute the basic branches (in $m$) at fixed $(h,\eps)=(1,0.15)$ 
and $a=1.05$; the latter yields a small deviation from the circular 
cone to break the rotational invariance. 
\item In Fig.\ref{f5} we continue folds which occur 
in Fig.~\ref{f4} in $h$, which shows that the relation of T1 and T2 
solutions is similar to the cusp in Remark \ref{crem}. 
\item In Fig.\ref{f5b} we 
continue solutions at $m=0$ from Fig.~\ref{f4} 
to smaller $\eps$, essentially verifying the main analytical 
result that $u_{\eps}\to u_0$ for a limit function $u_0$, and  
$E_\eps(u_\eps)\to |I_0|$, as $\eps\to 0$.  
\item In Fig.~\ref{f6} we fix $\eps=0.1$ and $m=0$, and study the 
dependence of T1, T2 and T3 solutions (and their energies) on $h$ and $a$. 
\eci  

One main result is that at small $\eps>0$ we generally get results 
similar to Fig.~\ref{f7}(a), see for instance Fig.~\ref{f6}(b), 
but also important differences: At small $\eps>0$, a tip interface 
is a minimizer for sufficiently large $a$ (depending on $h$), and 
the winding T2 interface does not exist (is not a critical point), 
Fig.~\ref{f6}(e,f), but for $\eps\to 0$ the T2 interface ``appears''  
(bifurcates from the T1 branch) and becomes the minimizer.

\subsection{The basic branches at $(h,a,\eps)=(1,1.05,0.15)$}
\label{ncsec1} 
In Fig.~\ref{f4} we  run ``continuation in $m$''   
at fixed $(h,a,\eps)=(1,1.05,0.15)$. 
This means that $m$ (additional to the Lagrange multiplier $\lam$ for 
the mass constraint $\spr{u}=m$) 
is the free parameter, which may move back and forth, as all continuation 
is done in so--called arclength parametrization \cite{p2p}. 

The black branch {\tt hom} in the 
bifurcation diagram (BD) in (a$_1$) represents 
the homogeneous solutions $u\equiv m$ with $\lam{=}W'(m)$ and 
$E(u){=}\frac 1 {2\sig \eps}W(m)$. 
On {\tt hom} we find (at this relatively large 
$\eps$) 14 BPs up to $m{=}0$, and we follow the 
 branches bifurcating 
\bci 
\item at BP1 (b1, blue, containing type T1 {\em and} T2 at 
$m=0$, see below 
for further discussion), 
\item at BP2 
(b2, green, like b1 but rotated by 90$^\circ$), 
\item at BP3 (b3, red, containing 
type T3 at $m=0$), 
\item and at BP5 (b5, magenta); this is for illustration of just 
one ``higher order branch''. 
\eci 
In the following we shall focus 
on branches b1 and b3.  
In the BDs, 
thick lines indicate stable parts of branches (local minima), 
and thin lines indicate unstable parts (saddles). 
The BD in (a$_1$) and those following are essentially 
verbatim outputs of \pdep\ scripts, with some postprocessing 
to adjust labels, subsequently used as identifiers in solution plots. 
We start the labeling with A at $m=0$ in panel (a$_2$) 
as $m=0$ will be our 
interest in the following pictures, while labels I--K are 
for further illustration only. 
Open circles indicate detected BPs.  
In (a$_2$) we zoom in on the branches b1, b2 near $m=0$, and in 
(a$_3$) on b3 near $m=0$. 
In (b,c) we give sample solutions as labeled in (a), and in (d) 
we present (a$_1$) for $m\in(-0.8,-0.45)$ 
in another projection, namely $\max(u)-m$.% 
\footnote{While this more clearly shows the structure of solutions near 
bifurcation, in the following we restrict to plotting $E_\eps$, 
as this is the quantity of interest.} 
The {\tt b1} solutions fulfill $\ga_2$ (reflection in $y$), but not 
$\ga_1$, and hence {\tt b1} must bifurcate in a pitchfork. 
The bifurcation is subcritical but the branch stabilizes in a first 
fold near $m=-0.79$.  
On the other hand, solutions on {\tt b3} fulfill $\ga_1$ and $\ga_2$ 
and bifurcate transcritical, as expected. Here, the subcritical  
part (to more negative $m$) also folds back and stabilizes. 

\begin{figure}[h]
\bce
\btab{l}{
\btab{lll}{{\sm (a$_1$)}&{\sm (a$_2$)}&{\sm (a$_3$)}\\
\hs{-3mm}\ig[width=0.32\tew]{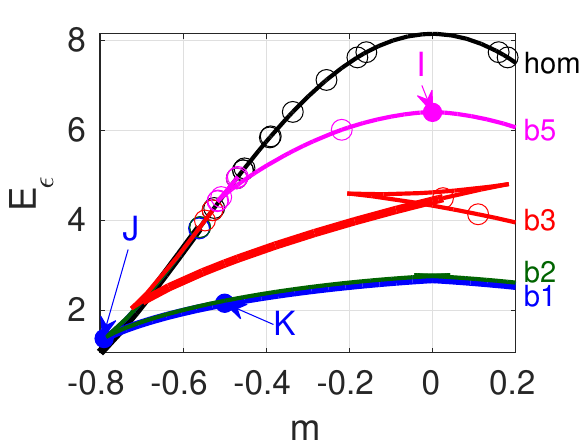}&\hs{-0mm}\ig[width=0.28\tew]{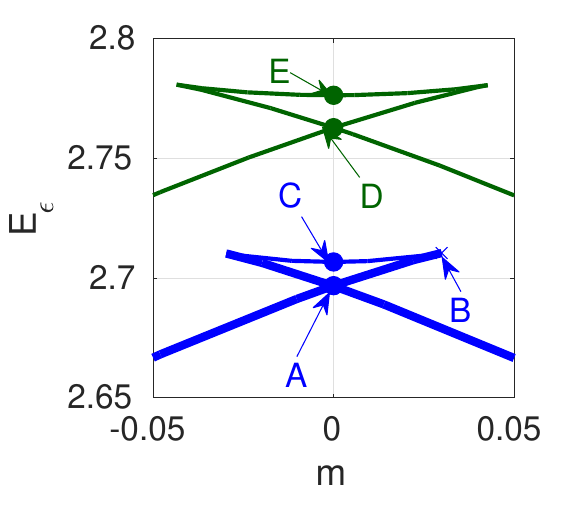}
&\hs{-3mm}\ig[width=0.35\tew]{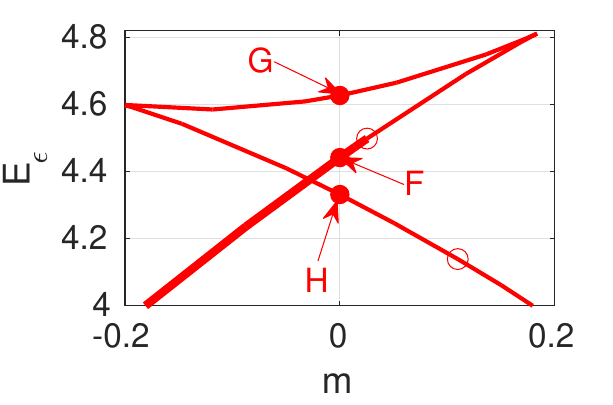}}\\
{\sm (b)}\\
\hs{-0mm}\ig[width=0.24\tew]{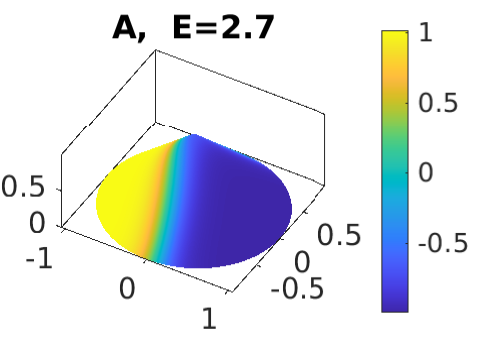}\hs{-0mm}\ig[width=0.14\tew]{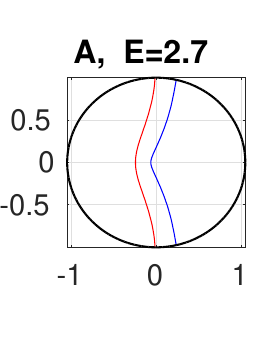}
\ig[width=0.14\tew]{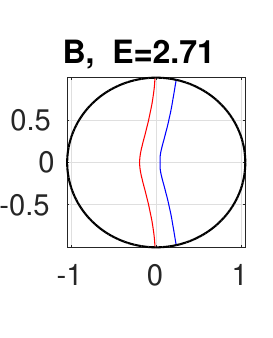}
\ig[width=0.14\tew]{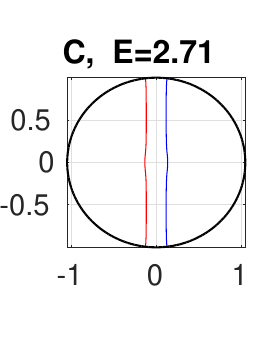}\ig[width=0.14\tew]{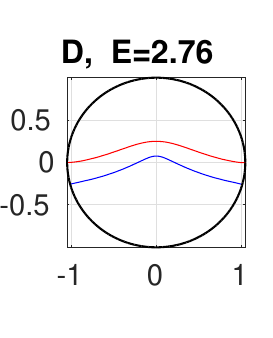}
\ig[width=0.14\tew]{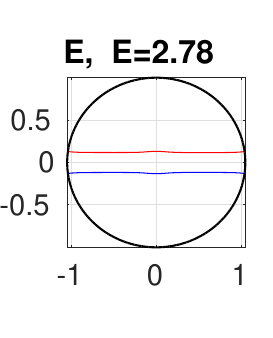}
\\
\btab{ll}{{\sm (c)}&{\sm (d)}\\
\rb{20mm}{\btab{l}{
\ig[width=0.14\tew]{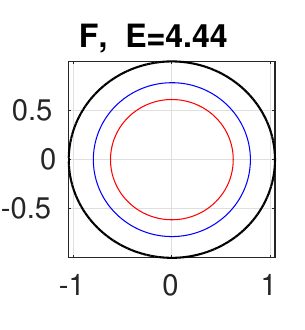}\ig[width=0.14\tew]{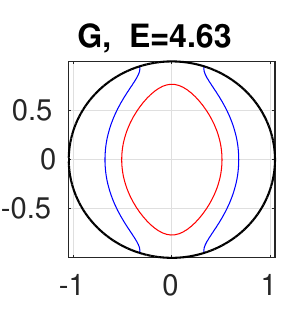}
\ig[width=0.14\tew]{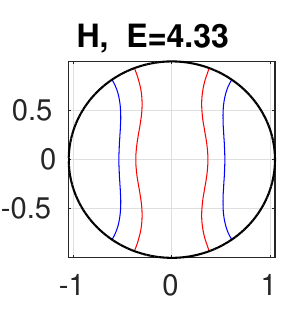}\\[-3mm]
\hs{-2mm}\rb{3mm}{\ig[width=0.22\tew]{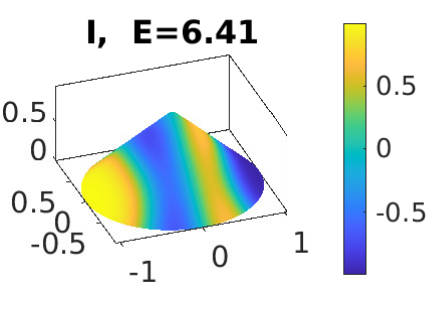}\hs{2mm}
\ig[width=0.22\tew]{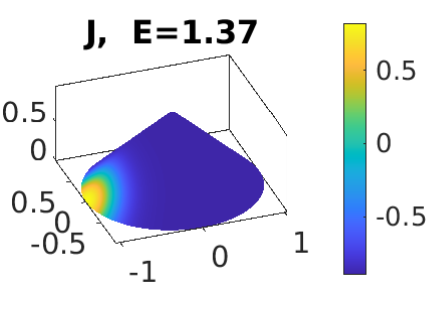}\hs{2mm}\ig[width=0.22\tew]{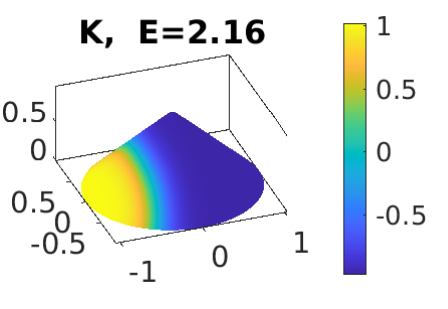}
}}
}
&\hs{-5mm}\ig[width=0.25\tew]{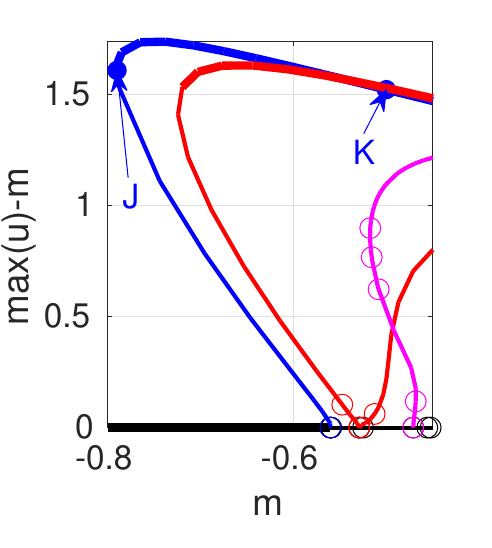}}
}
\ece\vs{-5mm}
\caption{{\small (a) BD for continuation in $m$, for $(h,a,\eps)=(1,1.05,0.15)$, 
with two zooms near 
$m{=}0$. 
Primary branches bifurcating from $u\equiv m$ (black branch hom 
in (a$_1$)): b1 (blue), b2 (green), b3 (red), and b5 (magenta). 
(b,c) Sample solutions  from (a), as indicated by labels. 
(d) same as (a$_1$), plotting $\max(u)-m$ over  
$m$ (near $m=-0.6$) for b1, b3, and b5. 
\label{f4}}}
\end{figure}

The first sample plot A in (b) shows the solution on the cone, 
and also gives the energy $E_\eps$ in the title.   
The second sample in (b) shows the contour lines $u=0.5$ (red) and 
$u=-0.5$ (blue) for the same solution A, with the base of the cone indicated 
in black, and subsequently we mostly use this plot--style. 
The sample B shows $u$ at the fold point at $m\approx 0.03$, which 
will be further discussed in Fig.~\ref{f5}. Together, A--C 
show that b1 approaches $m=0$ from the left as type T2 solution, 
and after the fold turns into a T1 solution C with a slightly 
larger energy. Subsequently, b1 continues symmetrically through the left fold 
and again past $m=0$ to the symmetric BPs on {\tt hom} at $m>0$. 
Essentially the same 
happens on b2, but rotated by 90$^\circ$ such that the 
T1 interface in E is along the $a=1.05$ semiaxis, with hence 
a (slightly) larger energy $E$ than in C. Thus, for b1 and b2 the T1 solution is related 
to the T2 solutions via folds, similar to the ``middle surface'' between 
the two fold--curves in Fig.\ref{cfig}, and in Fig.\ref{f5} we shall 
illustrate the cusp underlying this. 

The samples F--H in (c) show three passages of b3 through $m=0$. 
As already said, the bifurcation of b3 is 
transcritical as the solutions on it fulfill both, $\ga_1$ (reflection in $x$) 
and $\ga_2$ (reflection in $y$), but in (a$_1$) and (a$_3$) we only 
show the subcritical part; this returns supercritically 
to the symmetric BP at $m>0$. The most relevant sample 
on b3 at $m=0$ is the T3 solution F, as it has the lowest energy 
among F--H, and is in fact locally stable. 

Sample I in (c) 
belongs to b5. This is only meant as just one example of the 
many further solutions of \reff{ch2b}, containing more interfaces 
than the basic types T1, T2 and T3, and hence not expected to be 
energy minimizers also when varying other parameters. In particular, 
for $\eps\to 0$ such solutions may turn into so called $2n$--end types 
\cite{KLP15}, $n>1$, i.e., 
they may contain points where $2n\ge 4$ different phase domains 
$u=\pm 1$ meet. Specifically, sample I contains two four--end 
points for 
$\eps\to 0$, see I in Fig.\ref{f5b}(b) for illustration. 
Finally, samples J (at the first fold of b1) and K in panel (c) 
are meant to illustrate how b1 proceeds from negative $m$ to $m=0$ 
first reached at A; solutions on the other branches behave 
correspondingly. 

\subsection{Fold--continuation in $h$}\label{ncsec2}
In Fig.\ref{f5} we show the continuation of the fold 
at $m=m_0\approx 0.03$ on b1 (sample B in Fig.\ref{f4}). 
 In brief, this illustrates the 
cusp structure of the relation between T1 and T2 solutions. 
In detail, on the part 
containing B in Fig.\ref{f5}, the fold location increases with $h$, 
see sample B$^+$, which also 
implies that the T1 and T2 solutions separate more strongly with increasing $h$. 
Conversely, decreasing $h$ from $h=1$ the fold position goes to $m=0$, 
where it continues as the symmetric fold to sample ${\rm B}^-$. 
In particular, 
the collision of the two folds at the cusp $(m,h)=(0,h_c)$, $h_c\approx 0.8$, 
means that for $h<h_c$ we only have one solution of type T1 or T2, and by 
symmetry this must be of type T1 (straight interface). This is an important 
difference to Fig.\ref{f7}(b), where at $\eps=0$ the type T2 interface 
exists for all $h>0$ and coincides with T1 only in the limit $h\to 0$. 
The cusp is thus a finite $\eps$--effect, and indeed moves to smaller $h_c$ 
for decreasing $\eps$ (not shown). A further consequence, discussed 
in Fig.\ref{f6} below,  is that for finite $\eps$,  
when varying $h$ at $m=0$  the T2 solutions must bifurcate from the T1 branch 
at $h=h_c$. 
 
\begin{figure}[h]
\bce
\ig[width=0.17\tew]{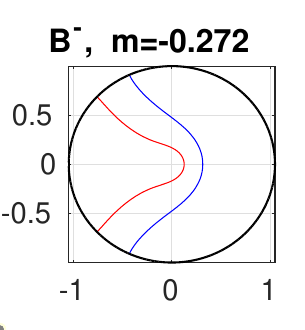}\quad\ig[width=0.24\tew]{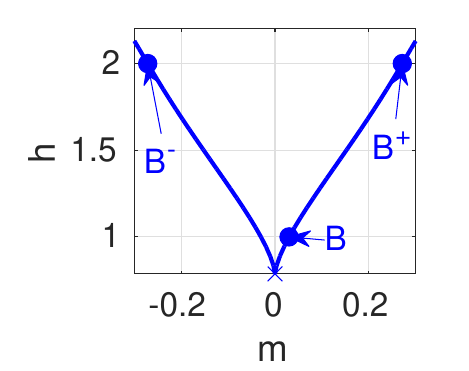}\quad
\ig[width=0.17\tew]{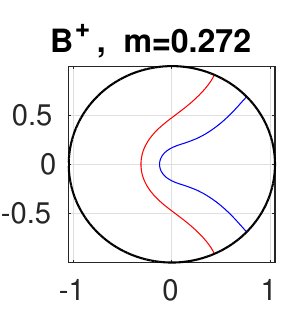}
\ece\vs{-4mm}
\caption{{\small Continuation of the fold position $m$ of the 
fold B from Fig.\ref{f4} in height of 
cone $h$,  
$(a,\eps)=(1.05,0.15)$. 
\label{f5}}}
\end{figure}

\subsection{Continuation in $\eps$}\label{ncsec3}
In Fig.\ref{f5b} we continue selected solutions from Fig.\ref{f4} to smaller 
$\eps$, aiming to verify that $E_\eps(u_\eps)\to |I_0|$ as $\eps\to 0$. 
Decreasing $\eps$ is numerically challenging as it requires repeated and 
strong mesh refinement near the interface. In the BD in (a) 
we show $E_\eps$ over $\eps$ for the continuation of the T2 sample A from 
Fig.\ref{f4} (orange branch), and for the T1 sample C 
from Fig.\ref{f4} (blue branch). 
The blue branch (straight 
interface along the short semi-axis of length 1) should 
limit to $2\sqrt{2}\approx 2.82$ (independent of $a$), 
and $E_\eps(u_\eps)$ 
initially strongly increases and reaches $E_\eps(u_\eps)\approx 2.795$ 
at $\eps=0.025$. 
However the numerics become increasingly harder at small $\eps$, 
and we stop at $\eps=0.025$, where we have adaptively refined 
from $n_t\approx 6000$ (at $\eps=0.15$) to 
$n_t\approx 60000$ triangles in the FEM mesh. 

\begin{figure}[h]
\bce
\btab{ll}{{\sm (a)}&{\sm (b)}\\
\hs{-3mm}\rb{-3mm}{\ig[width=0.27\tew]{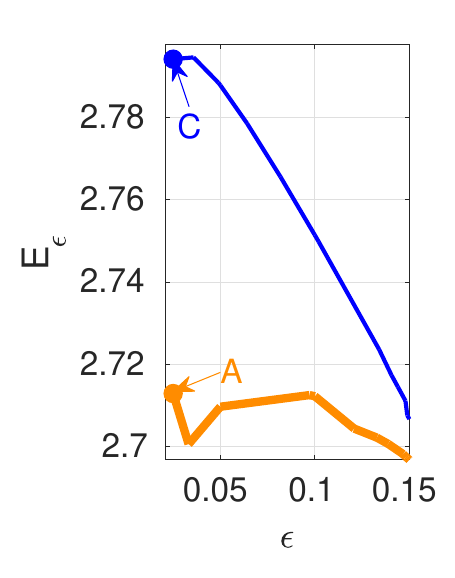}}
\rb{23mm}{\btab{l}{
\ig[width=0.18\tew]{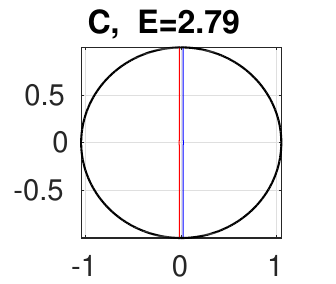}\hs{2mm}\\[-4mm]
\ig[width=0.18\tew]{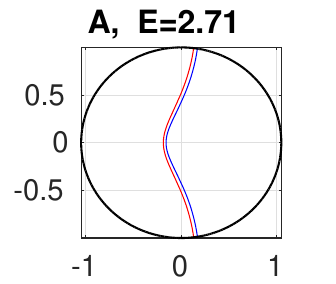}
}}
&\hs{-5mm}\ig[width=0.24\tew]{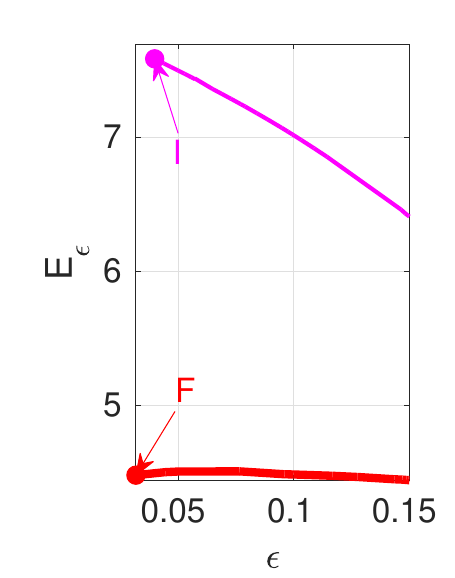}
\rb{23mm}{\btab{l}{
\ig[width=0.21\tew]{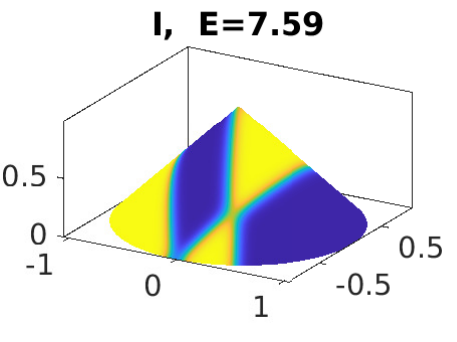}
\\[-2mm]
\ig[width=0.18\tew]{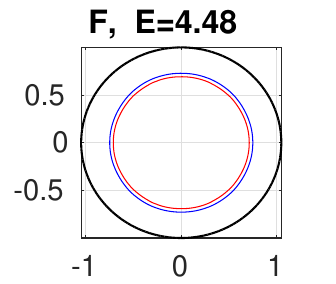}}
}}
\ece\vs{-5mm}
\caption{{\small (a)  Continuation in $\eps$ of the T2 (orange branch) 
and T1 (blue branch) 
solutions at $m=0$ from Fig.\ref{f4}(a$_2$). (b) same for solutions F and I 
from  Fig.\ref{f4}. 
\label{f5b}}}
\end{figure}

Note that, as observed in \cite{Ton05}, the minimizing interfaces
in dimension $2$ cannot have self-intersections. This does not contradict
Figure \ref{f5b} (b), since the illustrated interfaces are critical points but not minimizers. 

For $a=1$, the orange branch should limit to $E_0\approx 2.71$, 
see Fig.\ref{f7}(a). This will not 
change significantly for small $a{-}1>0$ as the interface is close to 
the short semi-axis of length 1 (see also Fig.~\ref{f6} below for general 
dependence of branches on $a$), and with 
the refinement to $n_t\approx 35000$ in A at $\eps=0.025$ we believe 
that we obtain a good approximation of $l_2(h)$.  
The sample plots show the expected behavior of the solutions, i.e., 
the interfaces (of width $\eps$) steepen up and hence the $u=\pm 0.5$ 
level lines in the samples move closer together. 
In (b), sample F shows the continuation of F from Fig.\ref{f4}  to 
$\eps=0.04$, and I the analogue for I from Fig.\ref{f4}, 
illustrating the two 4--end points 'expected' (see Remark \ref{sirem}b))  
in the limit $\eps\to 0$.

\brem\label{sirem}{\rm a) Theorem \ref{mthm1} does {\em not} apply 
to the blue branch, but down to $\eps=0.025$ the numerics suggest the 
convergence of $u_\eps$ to the tip--interface $I_0$, see also 
Remark \ref{ueexrem}(b) and the discussion after Fig.\ref{f1}. 
However, for smaller $\eps$ it becomes difficult 
to maintain the orientation of the interface in C, i.e., depending on the 
mesh small rotations of the interface may set in, and thus we stop the 
continuation. In any case, {\em if} we assume the convergence of 
$u_\eps\to u_0$ for $\eps\to 0$, 
then the blue branch shows that the convergence is {\em slower} than for 
branches 
(such as T2 and T3) on which interfaces avoid the tip. Moreover, 
this effect becomes stronger for more pointed cones, see 
Fig.~\ref{f6}(c), where we discuss the behavior of the T1 interface 
in dependence of $\eps$ in more detail for a cone of height $h=3$. 

b) We also have no proof of convergence for the interfaces 
on the magenta branch containing I, which again contains 
a straight segment through the tip. Other ``higher order branches'' 
with solutions with several phase domains but for which the interfaces 
avoid the tip show a better convergence behavior, similar to the branches 
containing A and F. 
}\eex\erem

\subsection{Continuation in $h$ and $a$, and again in $\eps$}
\label{ncsec4}
In Fig.\ref{f6} we aim to study the dependence of the three main 
types of solutions on $a$ and $h$. 
We initially fix the ``intermediate'' $\eps=0.1$, and in (a) 
we restart similar to Fig.\ref{f4} 
with the continuation in $m$ at $h=0.25$. As expected from Fig.\ref{f5}, 
the primary bifurcating branch b1 (blue) then goes horizontally through 
$m=0$ with a T1 type solution at $m=0$, i.e., the loop containing 
B and C in Fig.\ref{f4}(a$_2$) no longer exists. 

\begin{figure}[h]
\bce
\btab{l}{
\btab{ll}{
{\sm (a)}&{\sm (b)}\\ 
\hs{-2mm}\ig[width=0.23\tew]{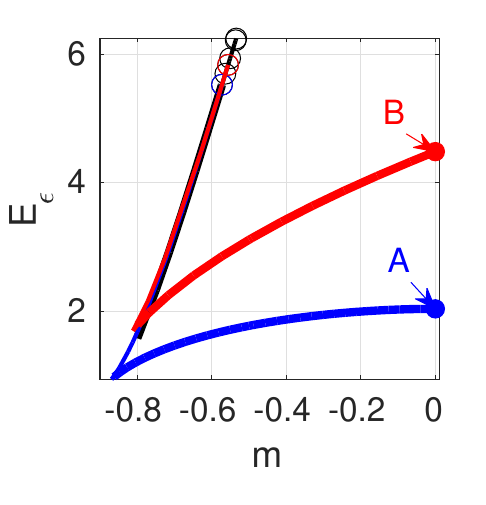}
\rb{14mm}{\btab{l}{\ig[width=0.13\tew]{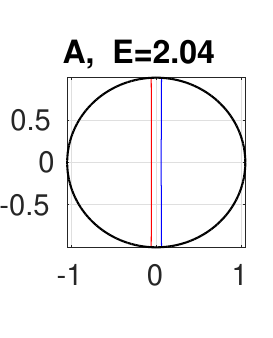}\\[-5mm]
\ig[width=0.13\tew]{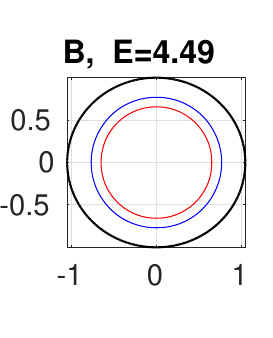}}}&
\hs{-4mm}\rb{-5mm}{\ig[width=0.26\tew]{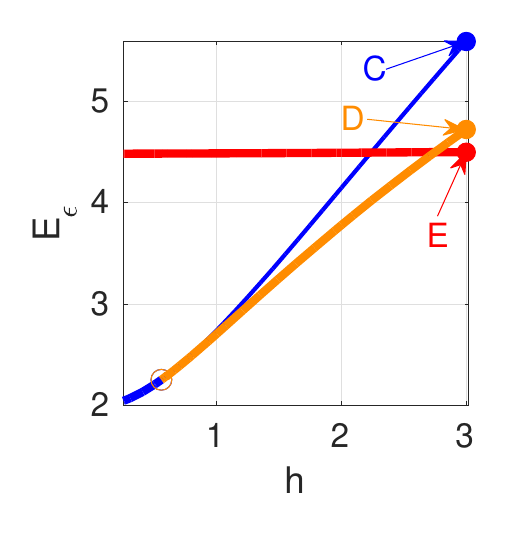}}
\hs{-4mm}\ig[width=0.22\tew]{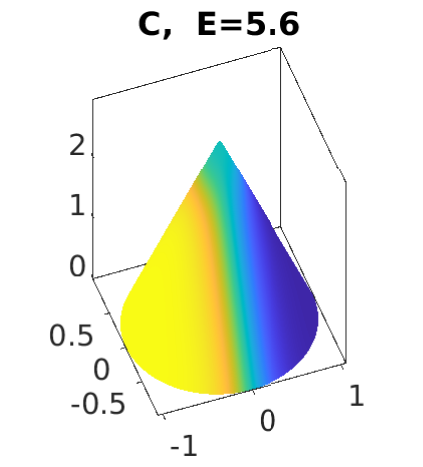}
\hs{-4mm}\rb{14mm}{\btab{l}{\ig[width=0.13\tew]{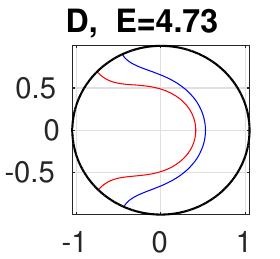}\\[0mm]
\ig[width=0.13\tew]{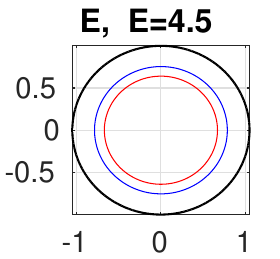}}}}
\\[-2mm]
\btab{ll}{{\sm (c)}&{\sm (d)}\\
\hs{-2mm}\ig[width=0.24\tew]{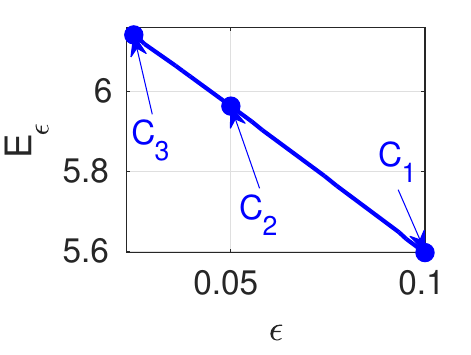}
\rb{2mm}{\ig[width=0.18\tew]{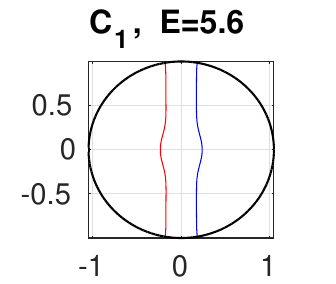}
\hs{-2mm}\ig[width=0.18\tew]{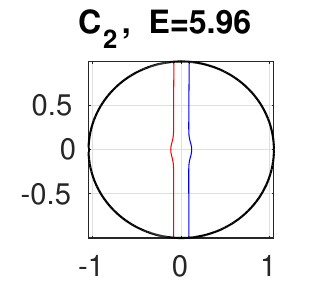}
\hs{-2mm}\ig[width=0.18\tew]{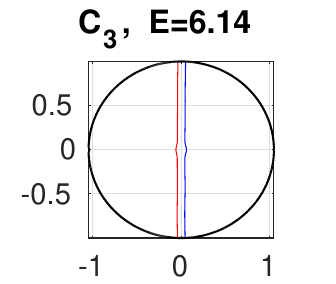}}
&\hs{-2mm}\ig[width=0.2\tew]{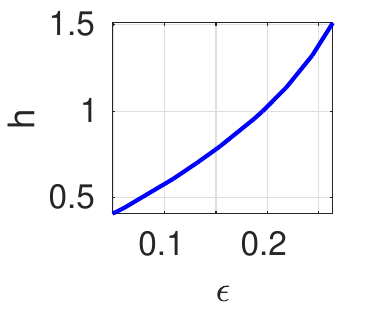}
}
\\
\btab{ll}{
{\sm (e)}&{\sm (f)}\\
\hs{-3mm}\ig[width=0.24\tew]{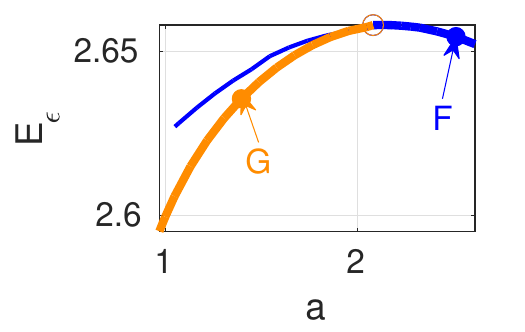}\hs{0mm}
&\rb{0mm}{\ig[width=0.27\tew]{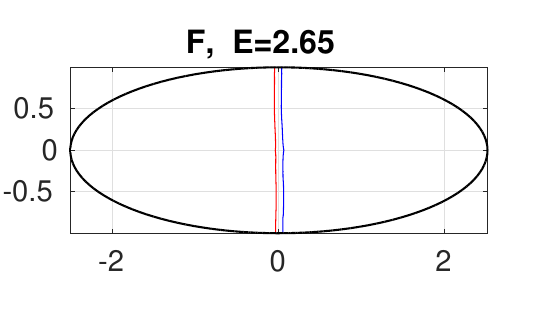}
\hs{-3mm}\ig[width=0.25\tew]{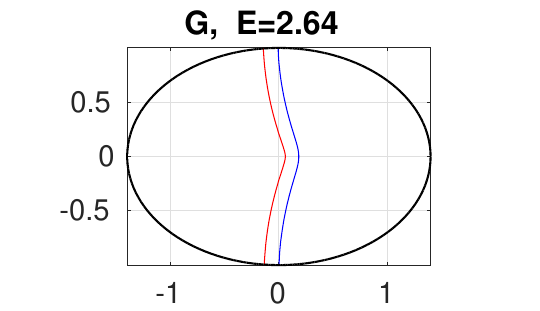}\hs{-5mm}\ig[width=0.25\tew]{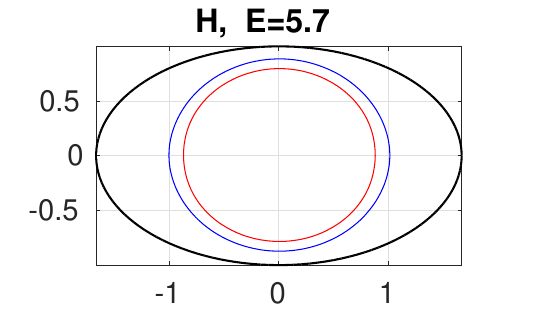}}
}}
\ece\vs{-3mm}
\caption{{\small (a) T1 and T3 branches at $(h,a,\eps)=(0.25,1.05,0.1)$. 
(b) Continuation of T1 and T3 at $m=0$ from (a) in $h$. The T2 
branch (magenta) bifurcates at $h=h_c=\approx 0.72$ from T1. 
The T3 solutions are global minima for $h\ge h_0\approx 2.7$.   
(c) Continuation of C from (b) in $\eps$, samples at $\eps=0.1$ (same as C), 
at $\eps=0.05$, and at $\eps=0.025$. 
(d) Branch point continuation of the BP at $h_c$ from (b). 
(e,f) continuation of T1 from (b) at $h=0.9$ in $a$. T2 reconnects 
to $T_1$ at $a=a_0\approx 2.2$. The last plot in (d) also shows 
a continuation of T3 from (a) in $a$. 
\label{f6}}}
\end{figure}

In (b) we then continue 
samples A and B from (a) in $h$. For the T3 (red) branch, 
 the energy $E_\eps$ compares reasonably well with the length 
$l_3$ from Fig.\ref{f7}(a), where again we note that the deviation 
due to finite $\eps=0.1$ is larger than the one due to $a-1=0.05>0$. 
Namely, for the T3 branch $E$ should not depend 
on $h$, and should be close to $\sqrt{2}\pi\approx 4.44$ 
(the limit $\eps\to 0$ for $a=1$), and this holds reasonably well. 
On the other hand, for the tip--interface branch branch T1 (blue) 
$E=5.6$ in sample C at $h=3$, while 
$l_1(3)=2\sqrt{1+3^2}\approx 6.32$ (independent of $a$), i.e., the 
more pointed tip here induces a strong deviation even at relatively 
small $\eps=0.1$. Therefore, in (c) we continue C from (b) to smaller 
$\eps$, where at $\eps=0.025$ in C$_3$ we have reached $E_\eps=6.14$. 
The contour plots in (c) show that for tip--interfaces 
the level lines $u=\pm \del\ne 0$ (here $\del=0.5$) detour the 
tip for $\eps>0$, and this becomes more pronounced 
for more pointed cones (i.e., larger $h$, 
compare samples C,E in Fig.\ref{f4} with $(h,\eps)=(1,0.15)$, C in Fig.\ref{f5b} with $(h,\eps)=(1,0.05)$, 
A in Fig.\ref{f6}(a) with $(h,\eps)=(0.25,0.1)$, and C$_1$ to C$_3$ with 
$h=3$ and $\eps=0.1, 0.05, 0.025$). Thus, this effectively 
quantifies the effect of the strength of the conical singularity, 
yielding a slower convergence. 

The most important difference between Fig.\ref{f6}(b) and 
Fig.\ref{f7}(a) is that the area--halving circular arc  in Fig.\ref{f7}(a)  
exists for any $\al\in[0,2\pi)$, and becomes straight for 
$\ds\al=\frac{2\pi}{\sqrt{1+h^2}}\to 2\pi$, 
i.e., $h\to 0$. On the other hand, the T2 solutions in Fig.~\ref{f6}(b$_2$) 
with $\eps=0.1$ 
only exist for $h\ge h_c\approx 0.7$, where the (orange) T2 branch  
bifurcates from the T1 branch in a supercritical pitchfork. 
Again, this is due to the finite $\eps$, and $h_c$ decreases 
with $\eps$: For $\eps=0.15$ 
we know from Fig.\ref{f5} that the BP is given by the cusp of the 
fold point continuation and sits at $h\approx 0.8$, 
and in (d) we altogether show the BP location $h$ over $\eps$ as 
obtained from branch point continuation \cite[\S3.61.]{p2p}, where we 
need to stop near $\eps=0.04$ due to convergence problems. 

Finally, in (e,f) we show the continuation of the T1 (blue) and T2 (orange) 
solutions from (b) at $h=0.9$ in the ellipticity $a$.%
\footnote{We also continue 
the T3 solutions, but for these $E$ increases rather quickly from 4.5 
and the branch cannot be plotted reasonably together with the T1 and T2 
branches; hence we only give the sample H at $a\approx 1.7$, with 
$E\approx 5.7$.} 
For both, $E$ only depends weakly on $a$; in fact, for T1 
we should have $\lim_{\eps\to 0}E_\eps=2\sqrt{1+0.9^2}\approx 2.69$ 
independent of $a$, 
but as in Fig.\ref{f5} this requires extensive mesh-adaption, additional 
to the mesh--adaption already needed in Fig.\ref{f6}(d) for increasing $a$. 
Importantly, as we continue b1 we gain stability at a BP near 
$a=a_c\approx 2.2$, 
where the orange stable T2 branch bifurcates in a 
subcritical pitchfork. 
Alternatively, if we continue the T2 solution at $h=0.9$ from (b) in $a$, 
then the obtained branch coincides with the orange branch in (e), and 
hence ``reconnects''  
to the T1 branch at $a=a_c$. This intuitively makes sense as we expect 
the T2 solutions to approach the T1 solutions for large $a$, and 
the T1 solutions to become global minimizers. However, again this also 
depends on $\eps$, and if we decrease $\eps$ the BP (or reconnection point) 
in (e) moves to larger $a$ (not shown).  

This, and further numerical experiments fully agree with Prop.~\ref{prop:tip interface}, stating that tip interfaces are never minimizers at $\eps=0$: 
while for small $\eps>0$ and fixed $h>0$, 
T1 interfaces are the global minimizers for sufficiently large $a$, 
for $\eps\to 0$ at any fixed $a$ they lose stability to the T2 interfaces.

\appendix

\section{Appendix: Auxiliary aspects of geometric measure theory}\label{section-appendix}

%%%%%%%%%%%%%%%%%%%%%%%%%%%%%%%%%
\subsection{Functions of bounded variation and the co-area formula}
%%%%%%%%%%%%%%%%%%%%%%%%%%%%%%%%%

Functions of bounded variation are suitable for the study of our 
$\Ga$--convergence. 
They are defined on so-called {\em good} metric measure spaces, and 
there are two different gradient norms that can both be put into relation. 
With both gradient norms and the Hausdorff measure, 
one can define the perimeter and establish the co-area formula. Our main sources are \cite{Vol,Mir,Nic,FlRi}.

%%%%%%%%%%%%%%%%%%%%%%%%%%%%%%%%%
\subsubsection{Metric measure spaces and Hausdorff measure}
%%%%%%%%%%%%%%%%%%%%%%%%%%%%%%%%%

\begin{Def}
Let $(X,\mathfrak{d},\mu)$ be a metric measure space with metric $\mathfrak{d}$ and measure $\mu$. It is called {\em good} if
$(X,\mathfrak{d})$ is complete, and $(X,\mathfrak{d},\mu)$ is doubling, i.e. if $B_R(x)$ denotes a ball in $X$ of radius $R$ 
centered at $x$, then there exists a uniform constant $c > 0$ such that for any $x \in X$ and any $R>0$
$$\mu(B_{2R}(x)) \leq c\cdot \mu(B_R(x)).$$
\end{Def}

A central example of such a good metric measure space in our context is
a compact Riemannian manifold. The Riemannian metric defines the distance $\mathfrak{d}_g$ 
and the Riemannian Lebesgue measure $d\vol_g$. Then $(M,\mathfrak{d}_g, d\vol_g)$ is a good metric measure space
in the sense of the definition above.
More generally, compact manifolds with boundary and conical singularities, see Definition \ref{def:mfd con}, are good metric measure spaces.
 More generally, compact stratified spaces yield good metric measure spaces.

\begin{Def}[Hausdorff measure] ~\\
Let $(X,\mathfrak{d},\mu)$ be a metric measure space. For any subset $\Omega \subset X$, any $s\in [0,\infty)$ define the $s$-dimensional Hausdorff measure
\begin{equation*}\label{eq:HausdorffMeasureDef}
    \cH_d^s(\Omega):= \ds\sup_{\delta >0 }\left[ \inf \left\{ \frac{\pi^{s/2}}{\Gamma(3/2)}
    \ds\sum_{j=1}^\infty \left( \frac{\diam_{\mathfrak{d}} \Omega_j}{2}\right)^s: 
    \Omega\subset\bigcup\limits_{j=1}^\infty \Omega_j ;\, \diam_{\mathfrak{d}} \Omega_j \leq \delta \right\} \right], 
\end{equation*}
where $\diam_{\mathfrak{d}} \Omega_j$ is the diameter of the smallest metric ball containing $\Omega_j$, and $\Gamma(s)$ is the Gamma-function.
\end{Def}

The constant $\frac{\pi^{s/2}}{\Gamma(3/2)}$ ensures that the Hausdorff measure corresponds to the Lebesgue measure 
on smooth Riemannian manifolds $(M,g)$ if $s=\dim M$. In that case, the Hausdorff measure 
$\cH_g^s(\Omega)$ coincides with 
the Riemannian measure of a Borel set $\Omega \subset M$, cf. e.g. \cite[Theorem 2.17]{Vol}.

%%%%%%%%%%%%%%%%%%%%%%%%%%%%%%%%%
\subsubsection{Functions of bounded variation}
%%%%%%%%%%%%%%%%%%%%%%%%%%%%%%%%%
\begin{Def}[Gradient for locally Lipschitz functions] ~\\
Let $(X,\mathfrak{d},\mu)$ be a metric measure space. For any open $\Omega\subset X$ 
consider the space $\Lip_\loc (\Omega)$ of locally Lipschitz functions 
$u:{\Om}\rightarrow \R$. For such $u$ and any $x\in X$ we define
\begin{equation*}
    \lVert \nabla u \rVert (x):= \ds\liminf_{\rho \rightarrow 0}{\left[\ds\sup_{y\in \overline{B_\rho (x)}}\frac{\lvert u(x)-u(y)\rvert}{\rho}\right]}.
\end{equation*}
\end{Def}

\noindent One can define the gradient norm for locally integrable functions $L_\loc^1(\Omega)$
on some open $\Omega$. 

\begin{Def}[Gradient of locally integrable functions]\label{def:GradLocLip}

Let $(X,\mathfrak{d},\mu)$ be a metric measure space. Let $\Omega{\subset}X$ be open and bounded, and 
    $u{\in}L_\loc^1(\Omega)$. Then
    \begin{equation*}
    \begin{split}
    \lVert \nabla u \rVert (\Omega) &:= \inf \left\{ \ds\liminf_{n\rightarrow\infty}
    {\ds\int_\Omega \lVert \nabla u_n \rVert (x) \, d\mu (x)} : (u_n)_n\subset \Lip_\loc (\Omega);
    u_n\xrightarrow{L^1_\loc(\Omega)} u \right\} \\
    \end{split}
    \end{equation*}
    \end{Def}

\noindent Now, we can define functions of (locally) bounded variation.

\begin{Def}\label{def:BV}
The space of functions with (locally) bounded total variation is 
\begin{itemize}
    \item $BV(\Omega):= \left\{ u\in L_\loc^1(\Omega) : \norm{\nabla u} (\Omega) < \infty\right\}$
    \item $BV_\loc (\Omega) := \left\{ u \in L_\loc^1(\Omega) : \norm{\nabla u} (A) < \infty \text{ for any } A\subset \Omega \text{ open},\right. \\
        \left. \overline{A}\subset \Omega \text{ compact }\right\}$
\end{itemize}
\end{Def}

\noindent The gradient can be computed as a limit for an approximating sequence.

\begin{Prop}[Approximating $\norm{\nabla u}(\Omega)$] \label{prop:ApproxGrad} 

For any $u{\in}BV(\Omega)$ with $\norm{\nabla u}(\Omega)<\infty$ there exists $(u_n) \subset \Lip_\loc(\Omega)$ such that
    \begin{itemize}
        \item $u_n \longrightarrow u$ in $L_\loc^1(\Omega)$,
        \item $\norm{\nabla u}(\Omega)=\ds \lim_{n\rightarrow \infty}\ds\int_\Omega \norm{\nabla u_n}(x)
= \ds\lim_{n\rightarrow\infty} \norm{\nabla u_n}(\Omega)$.
    \end{itemize}

\end{Prop}

\begin{proof}
The statement can be found in \cite[p. 984]{Mir}. 
\end{proof}

We will also need the following result. 

\begin{theorem}[lower semi-continuity]\label{thm:lower semi-cont} 
Suppose $(u_n)\subset BV(\Omega)$, and we have convergence 
$u_n\longrightarrow u$ in $L_\loc^1(\Omega)$ as $n\to \infty$. Then
\begin{equation}\label{eq:lower semi-cont}
    \norm{\nabla u}_g(\Omega)\leq \ds\liminf_{n\rightarrow\infty}{\norm{\nabla u_n}_g(\Omega)}.
\end{equation}
\end{theorem}

The proof of Theorem \ref{thm:lower semi-cont}, see e.g. \cite[Theorem 2.38]{Vol}, 
does not require any conditions on $\Omega$. In particular $\overline{\Omega}\subset M $ does not need to be compact.
\medskip

We conclude the section with a definition of perimeter.

\begin{Def}[Caccioppoli sets]\label{def:CacciSets}
Let $(X,\mathfrak{d},\mu)$ be a good metric measure space, arising from a Riemannian manifold $(M,g)$
of dimension $d$.
Let $E\subset M$ be a $\cH_{\mathfrak{d}}^{d}$-measurable set.
The perimeter of $E$ in $M$ is defined by 
\begin{equation*}
    P_g(E,\Omega):=\norm{\nabla \chi_E}(\Omega),
\end{equation*}
where $\chi_E:M\rightarrow\{0,1\}$ is the characteristic function of $E$.
If the perimeter is finite, $E$ is called a Caccioppoli set with respect to $\Omega$.
\end{Def}

This definition is compatible with our intuitive understanding of the perimeter
as the (Hausdorff) measure of the boundary. Namely, by 
e.g., \cite[Ch. 1]{Gi}, if $E\subset M$ is a Borel set with $C^1$-boundary $\partial E$, then $E$ is a Caccioppoli set with 
\begin{equation}\label{prop:PerimHausdorff}
    P_g(E,M)=\cH_{\mathfrak{d}}^{d-1}(\partial E).
\end{equation}

%%%%%%%%%%%%%%%%%%%%%%%%%
\subsection{$\Gamma$-convergence}
%%%%%%%%%%%%%%%%%%%%%%%%%
$\Gamma$-convergence is a powerful tool from the calculus of variations in order to study convergence of minimizers. 
We will recall the most important definitions and properties of $\Gamma$-convergence based on \cite[Ch. 13.1]{Ri}.

\begin{Def}[Abstract $\Gamma$-Convergence]\label{def:Gamma-conv}
Let $X$ be a complete metric space. The functional $\cF_\infty:X\rightarrow \R\cup\{+\infty\}$ is called 
(sequential) $\Gamma$-limit of the functionals $\cF_k:X\rightarrow\R\cup\{+\infty\}$ (denoted by 
$\cF_\infty:=\Gamma-\ds\lim_{k\rightarrow\infty}\cF_k$), $k\in\N$, if the following two conditions are satisfied.
\begin{itemize}
    \item For all sequences $(u_k)\subset X$ the " $\liminf$-inequality" holds
    \begin{equation}\label{eq:liminf-ineq}
    u =  \lim\limits_{k\rightarrow\infty}{u_k} \quad  \Longrightarrow  \quad 
        \cF_\infty[u]\leq \ds\liminf_{k\rightarrow\infty}{\cF_k[u_k]}
    \end{equation}
    \item For all $u\in X$ there exists as
     \textbf{recovery sequence} $(u_k)\subset X$ such that
     \begin{equation}\label{eq:recovery seq}
        u = \lim\limits_{k\rightarrow\infty}{u_k}, \quad \cF_\infty[u]=\ds\lim_{k\rightarrow\infty}\cF_k[u_k].
     \end{equation}
\end{itemize}
\end{Def}

\begin{Rem}\label{rem:limsup-ineq}
If the first condition in Definition \ref{def:Gamma-conv} is satisfied, then the second condition can be altered to the so 
called " $\limsup$-inequality": For all $u\in X$ there exists a sequence $u_k\longrightarrow u$ in $X$ as $k\rightarrow\infty$ such that
\begin{equation}\label{eq:limsup-ineq}
    \cF_\infty[u]\geq \ds\limsup_{k\rightarrow\infty}{\cF_k[u_k]}.
\end{equation}
\end{Rem}

An important consequence of $\Gamma$-convergence is that the limit
functional $\cF_\infty=\Gamma-\ds\lim_{k\rightarrow\infty}\cF_k$ is lower semi-continuous, 
cf.~\cite[Proposition 13.2]{Ri}. Moreover, $\Gamma$-convergence implies convergence of minima and the corresponding minimizers. 
More specifically we have the following results. 

\begin{theorem}\label{thm:conv_min}
Let $X$ be a complete metric space. Consider a sequence of functionals 
$\cF_k:X\rightarrow\R\cup\{+\infty\}$, $k\in \N$, and assume 
that the $\Gamma$-limit 
$\cF_\infty=\Gamma-\ds\lim_{k\rightarrow\infty}\cF_k$ exists. 

\begin{enumerate}
\item Assume that the functionals $\{\cF_k\}_{k\in\N}$ are equicoercive, 
i.e.~there exists a compact set $K\subset X$ such that
$\ds\inf_X \cF_k=\ds\inf_K \cF_k$ for all $k\in \N$.
Then, $\cF_\infty$ has a minimizer and 
\begin{equation*}
    \ds\min_X \cF_\infty= \ds\lim_{k\rightarrow \infty}\ds\inf_X \cF_k. 
\end{equation*}
In addition, all accumulation points of any precompact sequence $(u_k)\subset X$ with the property 
$\liminf_{k\rightarrow+\infty}\cF_k[u_k]=\liminf_{k\rightarrow+\infty}\inf_X \cF_k$ are minimizers of $\cF_\infty$.
\item Without the assumption of equicoercivity, 
a minimizer for $\cF_\infty$ need not exist a priori and we only have: If 
a sequence $\{u_k\}\subset X$ of minimizers  for $\{\cF_k\}$ 
converges to $u \in X$, then $u$ is a minimizer of 
$\cF_\infty$ and $\ds\lim_{k \rightarrow. \infty}\cF_k[u_k]=\cF_\infty[u]$.
\end{enumerate}
\end{theorem}
\begin{proof}
See e.g. \cite[Theorem 13.3]{Ri} for the first statement and \cite[Proposition 4]{Mo} for the second statement.
\end{proof}
\brem{\rm 
As noted in Remark \ref{ueexrem}, for our problem we can establish 
$u_\eps\to u_0$ in $L^1(M)$ via a priori estimates on $u_\eps$ in $H^1(M)$ 
and the compact embedding $H^1(M)\hookrightarrow L^1(M)$, 
which in fact yields the equicoercivity of our $E_\eps$. }
\eex\erem

%%%%%%%%%%%%%%%%%%%
\section{Conical singularities. $\mathscr{D}(\Delta_D), \mathscr{D}(\Delta_N)$ versus $H^2(M)$}\label{coneHD}

\begin{definition} \label{def:mfd con}
 A {\em compact manifold with boundary and conical singularity $P$} 
 is a metric space $\Mbar=M\cupdot\{P\}$
where $M$, the \textit{regular part}, is a smooth manifold, equipped with a Riemannian metric $g$, which admits a decomposition 
\[
M = \mathscr{C}(N) \cup_N X
\]
into a compact Riemannian manifold $X$ with disjoint boundary components $N$ and $\pa M$ and an open
truncated generalized cone $\mathscr{C}(N)$ over the closed manifold $N$. That is,
$\mathscr{C}(N)=(0,1]\times N$ and the metric on $\mathscr{C}(N)$ takes the 
special \lq conical\rq\ form
\begin{align*}
g|_{\mathscr{C}(N)}:=dr^2\oplus r^2g^N(r), \ r\in (0,1], 
\end{align*}
where $g^N(r)$ is a family of metrics on $N$ which is smooth up to $r=0$. 
The boundary $\{1\}\times N$ of $\mathscr{C}(N)$ is glued to the boundary component $N$ of $X$, and the metric on $\Mbar$ is defined by using the metric $d_M$ on $M$ induced by $g$ and setting, for $Q\in M$, 
$$d(P,Q):=\lim_{r\to0} d_M((r,y),Q)\quad \text{ for any } y\in N.$$
\end{definition}
\noindent
This means that the cone tip $\{ P \}$ corresponds to $r=0$.\footnote{See
  \cite{Gri:NDOCS} and \cite{MeWu:Geo} for more details on conical
  singularities, in particular for the proof that the intuitive notion
  of 'submanifold of $\R^m$ with conical singularity' satisfies this
  condition.}  $N$ is referred to as the \emph{cross section} of the
conical singularity.
 See Figure \ref{cone-fig} for an example with $\partial M=\emptyset$. We also speak of $M$ as a manifold with conical singularity $P$.
In a similar way we can admit several (finitely many)  conical singularities.

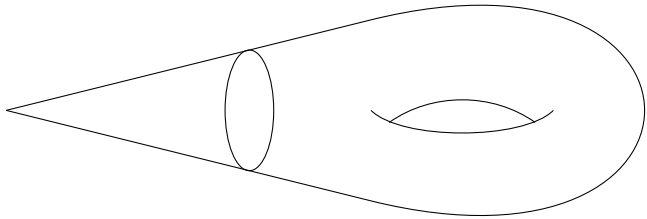
\begin{figure}[h]
\begin{center}
\begin{tikzpicture}[scale=1.6]
\draw (-2,4) -- (1,4.75);
\draw (-2,4) -- (1,3.25);
\draw (0,4) ellipse (0.2cm and 0.5cm);
\draw (1,4.75) .. controls (4,5.5) and (4,2.5) .. (1,3.25);
\draw (1,4) .. controls (1.25,3.75) and (2.25,3.75) .. (2.5,4);
\draw (1.15,3.9) .. controls (1.5,4.15) and (2,4.15) .. (2.35,3.9);
\end{tikzpicture}
\end{center}
\caption{Illustration of a manifold with a conical singularity.} \label{cone-fig}
\end{figure}

For a surface $\Mbar$, where $N=\mathbb{S}^1$, we define the
\textit{angle} of the conical singularity as
$\al=\lim_{r\to0} \frac{\ell_r}r$ where $\ell_r$ is the length of
$\mathbb{S}^1$ with respect to $g_{\mathbb{S}^1}(r)$. If $\Mbar$ is a straight cone
embedded in $\R^3$ with the conical tip at $P$ (i.e.\ $\Mbar$ is a union of straight segments starting at $P$, then $g_N$ is independent of $r$), then $\ell_r$ is the
length of the intersection of $\Mbar$ with a sphere of radius $r$ around
$P$, for small $r$.  The angle $\alpha$ also defines the opening angle
in a representation of $\Mbar$ as a subset of the plane (see Figure
\ref{f7}(b)).  
For the elliptic cone \reff{cpara} the intersection with
a small sphere is a convex curve contained in the open lower half
sphere, so $\al<2\pi$.  

We want to show here that
for $\alpha > 2\pi$
in spite of \eqref{DNH} generally
\begin{equation}\label{DNH2}\begin{split}
& \mathscr{D}(\Delta_D) \supsetneqq \{ u \in H^2(M) \mid \textup{tr} \, u = 0\}, \\
& \mathscr{D}(\Delta_N) \supsetneqq \{ u \in H^2(M) \mid \textup{tr} \circ \partial_\nu \, u = 0\}.
\end{split}\end{equation}
Let $(\lambda, \omega_\lambda)_{\lambda}$ be the set of eigenvalues and an orthonormal basis of 
corresponding eigenfunctions of the Laplace Beltrami operator $\Delta_N$ of $(N= \mathbb{S}^1,g_N)$.
Consider the minimal and maximal domains $\mathscr{D}(\Delta_{\min}), \mathscr{D}(\Delta_{\max})$
of the Laplacian, defined exactly as in \eqref{minmax-def} with $\nabla$ replaced by $\Delta$.
Classical arguments, see also e.g. some recent applications \cite[Lemma 2.2]{Vertman} or \cite{KLP:FDG}, 
show that for each $\omega \in \mathscr{D}(\Delta_{\max})$
there exist constants $c^\pm_\lambda(\omega)$ for all $\lambda \in [0,1)$, such that $\omega$ admits a partial asymptotic expansion as $r\to 0$
\begin{equation}\label{cone-asymptotics}
\begin{split}
\omega =& \sum_{\lambda = 0} \left(c^+_{\lambda}(\omega)
+ c^-_{\lambda}(\omega) \log(r) \right) \cdot \omega_\lambda 
\\ & + \sum_{\lambda \in (0,1)} \left(c^+_{\lambda}(\omega) r^{\sqrt{\lambda}}
+ c^-_{\lambda}(\omega) r^{-\sqrt{\lambda}} \right) \cdot \omega_\lambda 
 + \widetilde{\omega},
\end{split}
\end{equation}
where $\widetilde{\omega} \in \mathscr{D}(\Delta_{\min})$. The domains $\mathscr{D}(\Delta_D), \mathscr{D}(\Delta_N)$
are characterized by conditions on the coefficients $c^\pm_\lambda(\omega)$ that replace the usual boundary conditions: $\omega \in  \mathscr{D}(\Delta_{\max})$ is an element of 
$\mathscr{D}(\Delta_D)$ or $\mathscr{D}(\Delta_N)$ if and only if $c^-_\lambda(\omega) = 0$ for all $\lambda$, in addition to Dirichlet or Neumann
boundary conditions at the regular boundary, respectively. We can now give an example for \eqref{DNH2}.

\begin{example}{\rm Consider the conical singularity
    $(r,\theta) \in (0,1] \times \mathbb{S}^1$ with metric
    $dr^2 +  c^2r^2d\theta^2$, where $ c > 1$ and hence 
angle $\al=2\pi c>2\pi$. Here
    $\dim M \equiv m=2$. Then
    $(N,g_N) = (\mathbb{S}^1,  c^2 d\theta^2)$ with eigenvalues
    $\lambda$ given by $k^2/ c^2, k\in \Z$. Consider a cutoff
    function $\phi \in C^\infty[0,1]$ with $\phi \equiv 1$ near $r=0$
    and $\phi\equiv 0$ near $r=1$, and the first non zero
    eigenvalue $1/ c^2$ with eigenfunction
    $\omega_{1/ c^2}$. Then
    $\omega(r,\theta) := r^{1/ c}
    \omega_{1/ c^2}(\theta)\phi(r)$ lies in
    $\mathscr{D}(\Delta_D)$ and $\mathscr{D}(\Delta_N)$.  However,
    $\partial_r^2 \omega = r^{-2+1/ c} \omega_{1/ c^2}$ near
    $r=0$, which is not in $L^2(M,g)$ since the volume element is
    $r\,dr d\theta$, and hence $\omega \notin H^2(M)$. }
\end{example}

\providecommand{\href}[2]{#2}

${}^*$ daniel.grieser@uni-oldenburg.de, 
University Oldenburg, Germany 

${}^\dagger$ sina.held@uni-oldenburg.de, \
University Oldenburg, Germany 

${}^\ddagger$ hannes.uecker@uni-oldenburg.de, \
University Oldenburg, Germany 

${}^\sharp$ boris.vertman@uni-oldenburg.de \
University Oldenburg, Germany


\begin{thebibliography}{KLPW15}

\bibitem[\textsc{BL92}]{Hilbert-complexes}
J.~Br\"uning and M.~Lesch.
 \emph{Hilbert complexes.}
  J. Funct. An., 108,  pp. 88--132, 1992.

\bibitem[\textsc{BNAP22}]{BNAP22}
V.~Benci, S.~Nardulli, L.~Acevedo, and P.~Piccione.
 \emph{Lusternik-{S}chnirelman and {M}orse theory for the van der
  {W}aals--{C}ahn--{H}illiard equation with volume constraint.}
  Nonlinear Anal., 220, Paper No. 112851, 29, 2022.

\bibitem[\textsc{Che83}]{Che2}
J.~Cheeger.
 \emph{Spectral geometry of singular {R}iemannian spaces.}
  J. Differential Geom., 18 (4),  pp. 575--657 (1984), 1983.

\bibitem[\textsc{CHL10}]{CHL10}
Xinfu Chen, D.~Hilhorst, and E.~Logak.
 \emph{Mass conserving {A}llen-{C}ahn equation and volume preserving mean
  curvature flow.}
  Interfaces Free Bound., 12, 2010.

\bibitem[\textsc{CR71}]{CR71}
M.~G. Crandall and P.~H. Rabinowitz.
 \emph{Bifurcation from simple eigenvalues.}
  J. Funct. Anal., 8,  pp. 321--340, 1971.

\bibitem[\textsc{DF20}]{DF20}
Q.~Du and X.~Feng.
 \emph{The phase field method for geometric moving interfaces and their
  numerical approximations.}
 In  Handbook of Num. Anal., vol. 21, pages 425--508.
  Elsevier, 2020.

\bibitem[\textsc{Ell89}]{Ell89}
C.~M. Elliott.
 \emph{The {C}ahn-{H}illiard model for the kinetics of phase separation.}
 In  Math. models for phase change problems (\'{O}bidos,
  1988), vol. 88 of  Internat. Ser. Numer. Math., pp. 35--73.
  Birkh\"{a}user, Basel, 1989.

\bibitem[\textsc{ESS92}]{ESS92}
L.~C. Evans, H.~M. Soner, and P.~E. Souganidis.
 \emph{Phase transitions and generalized motion by mean curvature.}
  Comm. Pure Appl. Math., 45(9),  pp. 1097--1123, 1992.

\bibitem[\textsc{Eva98}]{evans}
L.C. Evans.
 \emph{Partial Differential Equations.}
 AMS, 1998.

\bibitem[\textsc{Fed14}]{Fe14}
H.~Federer.
  \emph{Geometric measure theory.}
 Springer, 2014.

\bibitem[\textsc{FR60}]{FlRi}
W.~Fleming and R.~Rishel.
 \emph{An {I}ntegral {F}ormula for {T}otal {G}radient {V}ariation.}
  Archiv Math., 11 (1),  pp. 218--222, 1960.

\bibitem[\textsc{GG18}]{GG}
P.~Gaspar and M.~Guaraco.
 \emph{The {A}llen--{C}ahn equation on closed manifolds.}
  Calc. Var. PDE. 57 (4),  pp. 1--42, 2018.

\bibitem[\textsc{GHP03}]{GHP03}
C.~E. Garza-Hume and P.~Padilla.
 \emph{Closed geodesics on oval surfaces and pattern formation.}
  Comm. Anal. Geom., 11 (2),  pp. 223--233, 2003.

\bibitem[\textsc{Giu84}]{Gi}
E.~Giusti.
 \emph{Minimal Surfaces and Functions of Bounded Variation.}
 Monographs in Math. 80. 1984.

\bibitem[\textsc{GM88}]{GM88}
M.~E. Gurtin and H.~Matano.
 \emph{On the structure of equilibrium phase 
transitions within the gradient theory of fluids.}
  Quart. Appl. Math., 46(2),  pp. 301--317, 1988.

\bibitem[\textsc{Gri11}]{Gri:NDOCS}
D.~Grieser.
 \emph{A natural differential operator on conic spaces.}
  Discrete Contin. Dyn. Syst., Dynamical systems, Diff. Eq. and Appl.
8th AIMS Conference. Suppl. Vol. I,  pp. 568--577, 2011.

\bibitem[\textsc{GS02}]{GoS2002}
M.~Golubitsky and I.~Stewart.
 \emph{The symmetry perspective.}
 Birkh\"auser, Basel, 2002.

\bibitem[\textsc{Hoy06}]{hoyle}
R.B. Hoyle.
 \emph{Pattern formation.}
 Cambridge Univ. Press., 2006.

\bibitem[\textsc{HT00}]{HuTo00}
J.~Hutchinson and Y.~Tonegawa.
 \emph{Convergence of phase interfaces in the van der
  {W}aals-{C}ahn-{H}illiard theory.}
  Calc. Var. PDE., 10 (1),  pp. 49--84, 2000.

\bibitem[\textsc{KLP08}]{KLP:FDG}
K.~Kirsten, P.~Loya, and J.~Park.
 \emph{Functional determinants for general self-adjoint extensions of
  {L}aplace-type operators resulting from the generalized cone.}
  Manuscripta Math., 125 (1),  pp. 95--126, 2008.

\bibitem[\textsc{KLPW15}]{KLP15}
M.~Kowalczyk, Yong Liu, F.~Pacard, and Juncheng Wei.
 \emph{End-to-end construction for the {A}llen-{C}ahn equation in the plane.}
  Calc. Var. PDE., 52 (1-2),  pp. 281--302, 2015.

\bibitem[\textsc{Kuz04}]{kuz}
Y.~A. Kuznetsov.
 \emph{Elements of applied bifurcation theory.}
 Springer, 3d edition, 2004.

\bibitem[\textsc{LP61}]{Lumer}
G.~Lumer and R.S. Phillips.
 \emph{Dissipative operators in a Banach space.} 
 Pacific J. Math., (11) 1961.

\bibitem[\textsc{Lun95}]{Lun95}
A.~Lunardi.
 \emph{Analytic semigroups and optimal regularity in parabolic
  problems.}
 Birkh\"{a}user, Basel, 1995.

\bibitem[\textsc{Mir03}]{Mir}
M.~Miranda.
 \emph{Functions of bounded variation on ``good'' metric spaces.}
  J. Math. Pures Appl. 82 (8),  pp. 975--1004, 2003.

\bibitem[\textsc{Mir19}]{Mira19}
A.~Miranville.
 \emph{The {C}ahn-{H}illiard equation, volume~95 of  CBMS-NSF
  Regional Conference Series in Applied Mathematics.}
 SIAM, PA, Recent advances and appl., 2019.

\bibitem[\textsc{Mod87}]{Mo}
L.~Modica.
 \emph{The {G}radient {T}heory of {P}hase {T}ransitions and the {M}inimal
  {I}nterface {C}riterion.}
  Arch. Rat. Mech. and Anal., 98 (2),  pp. 123--142,
  1987.

\bibitem[\textsc{Mor16}]{Morgan}
F.~Morgan.
 \emph{Geometric measure theory: a beginner's guide.}
 Acad. Press, 2016.

\bibitem[\textsc{MV12}]{Vertman}
R.~Mazzeo and B.~Vertman.
 \emph{Analytic {T}orsion on {M}anifolds with {E}dges.}
  Adv. in Math., 231 (2),  pp. 1000--1040, 2012.

\bibitem[\textsc{MW04}]{MeWu:Geo}
R.~Melrose and J.~Wunsch.
 \emph{Propagation of singularities for the wave equation on conic
  manifolds.}
  Invent. Math., 156 (2),  pp. 235--299, 2004.

 \bibitem[\textsc{Nic11}]{Nic}
L.~Nicolaescu.
 \emph{The {C}o-area {F}ormula.}
available at www3.nd.edu/~lnicolae/Coarea.pdf, 2011.

\bibitem[\textsc{Pac12}]{Pa12}
F.~Pacard.
 \emph{The role of minimal surfaces in the study of the {A}llen-{C}ahn
  equation.}
 In  Geometric analysis, Partial Differential Equations and
  surfaces, vol. 570 of  Contemp. Math., pp. 137--163. Amer. Math.
  Soc., Providence, RI, 2012.

\bibitem[\textsc{Rin18}]{Ri}
F.~Rindler.
 \emph{Calculus of Variations.}
 Springer, 2018.

\bibitem[\textsc{RS13}]{RS13}
N.~Roidos and E.~Schrohe.
 \emph{The {C}ahn-{H}illiard equation and the {A}llen-{C}ahn equation on
  manifolds with conical singularities.}
  Comm PDE, 38 (5),  pp. 925--943, 2013.

\bibitem[\textsc{Ton05}]{Ton05}
Y.~Tonegawa.
 \emph{On stable critical points for a singular perturbation problem.}
  Comm. Anal. Geom., 13 (2),  pp. 439--459, 2005.

\bibitem[\textsc{Uec21}]{p2p}
H.~Uecker.
 \emph{Numerical Continuation and Bifurcation in Nonlinear PDEs.}
 SIAM, Philadelphia, PA, 2021.


\bibitem[\textsc{Uec23a}]{spctut}
H.~Uecker. \emph{Continuation of fold points, branch points, and {H}opf points with
constraints in pde2path.} available at \cite{p2phome}, 2023.

\bibitem[\textsc{Uec23b}]{p2phome}
H.~Uecker. available online at \\ 
www.staff.uni-oldenburg.de/hannes.uecker/pde2path, 2023.


\bibitem[\textsc{Ver16}]{BV16}
B.~Vertman.
 \emph{The biharmonic heat operator on 
edge manifolds and non-linear fourth order equations.}
 Manuscripta Math, (146), pp. 179--203, 2016.

 
\bibitem[\textsc{Vol10}]{Vol} A.~Volkmann.
 \emph{Regularity of isoperimetric hypersurfaces with obstacles in
  Riemannian manifolds.} Diplomarbeit, Albert-Ludwigs-Universit\"at Freiburg,
2010.

\end{thebibliography}
\end{document}